\documentclass[10pt]{IEEEtran}

\usepackage{natbib}        % required for bibliography

\usepackage{multirow}
\usepackage{blkarray}
\usepackage{array}
\usepackage[symbol]{footmisc} 

\usepackage{amsmath,amssymb,amsfonts}
\usepackage{algorithmic}
\usepackage{graphicx}
\usepackage{textcomp}
\usepackage{xcolor}
\usepackage{makecell}

\usepackage{mathtools} % table rows
\usepackage{multirow}% table rows

\usepackage{lipsum}

\newtheorem{theorem}{\textbf{Theorem}}
\newtheorem{corollary}{\textbf{Corollary}}

\newtheorem{lem}{\textbf{Lemma}}

\newtheorem{assumption}{\textbf{Assumption}}
\newtheorem{remark}{\textbf{Remark}}

\usepackage{epstopdf}

\def\BibTeX{{\rm B\kern-.05em{\sc i\kern-.025em b}\kern-.08em
    T\kern-.1667em\lower.7ex\hbox{E}\kern-.125emX}}

\newcommand\blfootnote[1]{%
\begingroup
\renewcommand\thefootnote{}\footnote{#1}%
\addtocounter{footnote}{-1}%
\endgroup
}

\begin{document}

%%%%%%%%%%%%%%%%%%%%%%%%%%%%%%%%%%%%%%%%%%%%%%%%%%%%%%%%%%%%%%%%%%%%%%%%%%%%%%%%%%%%%%%%%%%%%%%%%%%%%%%%%%%%%%%%%%%%%%%%%%%%%%%%%%%%%%%%%%%%%%%%%%%%%%%%%%%%%%%%%%%%%%%%%%%%%%%%%%%%%%%%%%%%

%\begin{frontmatter}
\title{%Discrete-time
%Unbiased
Extremum Seeking of Static Maps in the Presence of  Unknown Large Time-Varying Delays}

\author{ Adam Jbara$^{a}$, Emilia Fridman$^{a}$ and Xuefei Yang$^{b,\dagger}$  }

%%%%%%%%%%%%%%%%%%%%%%%%%%%%%%%%%%%%%%%%

\date{Dated: \today}
\maketitle

%%%%%%%%%%%%%%%%%%%%%%%%%%%%%%%%%%%%%%%%
%\author{ Adam Jbara, Xuefei Yang,  Emilia Fridman ,
%\IEEEmembership{Fellow, IEEE}
%
%\thanks{This work was supported by Israel Science Foundation (grant no. 446/24), Chana and Heinrich Manderman Chair on System Control at Tel Aviv University.}
%\thanks{A. Jbara and E. Fridman are with the School of Electrical Engineering, Tel-Aviv University, Tel-Aviv, Israel (e-mail:  adamjbara@mail.tau.ac.il, emilia@tauex.tau.ac.il).} }

%\author[First]{Adam Jbara}
%\author[First]{Emilia Fridman}
%\author[Second]{Xuefei Yang}

%\address[First]{School of Electrical Engineering, Tel-Aviv University, Tel-Aviv, Israel\\
% (e-mail: adamjbara@mail.tau.ac.il, emilia@tauex.tau.ac.il),}
%\address[Second]{Center for Control Theory and Guidance Technology, Harbin Institute of Technology, Harbin, China\\
% (e-mail: yangxuefei@hit.edu.cn).}

%	\tnotetext[t1]{\color{red} This work was supported by Israel Science Foundation (grant no. 446/24), Chana and Heinrich Manderman Chair on System Control at Tel Aviv University, the Fundamental Research Funds for the Central Universities (Grant No.HIT.OCEF.2023007), and the Science Center Program of National Natural Science Foundation of China (Grant 62188101).}

%%%%%%%%%%%%%%%%%%%%%%%%%%%%%%%%%%%%%%%%%%%%%%%%%%%%%%%%%%%%%%%%%%%%%%%%%%%%%%%%%%%%%%%%%%%%%%%%%%%%%%%%%%%%%%%%%%%%%%%%%%%%%%%%%%%%%%%%%%%%%%%%%%%%%%%%%%%%%%%%%%%%%%%%%%%%%%%%%%%%%%%%%%%%

\begin{abstract}
In this paper, we present the discrete-time unbiased extremum seeking (ES) algorithm for $n$-dimensional ($n$D) static quadratic maps in the presence of unknown time-varying measurement delays bounded by  known constants which can be large. The existing ES results in the presence of large delays are usually confined to known constant or slowly-varying delays, which is restrictive. We provide the first ES algorithm, which is robust with respect to {\it unknown large time-varying delays}. Moreover, we achieve the unbiased exponential convergence. We manage with such delays by choosing  dithers with  frequencies of the order of $\sqrt{\varepsilon}$, where the small parameter $\varepsilon>0$ appears in the dynamics of the real-time estimator. As expected, larger delays lead to a slower convergence. We provide qualitative and quantitative results based on the averaging analysis via delay-free transformation. For the quantitative bounds on the controller parameters that ensure the exponential unbiased convergence of the ES system, we assume that the Hessian of the map is uncertain and lies within a known range. Differently from its continuous-time counterpart, the small parameter in the discrete-time case defines the decay rate of the  estimation error system, making a quantitative bound on this parameter particularly important. We present also constructive conditions for the practical stability of the classical ES system. Our results are semi-global for globally quadratic maps, while for locally quadratic static maps, we provide a bound on the region of convergence. Our analysis shows that appropriate ES parameters can be found for any large unknown time-varying bounded delay. A numerical example highlights the  efficiency of the method.

\end{abstract}

%\end{frontmatter}

%%%%%%%%%%%%%%%%%%%%%%%%%%%%%%%%%%%%%%%%%%%%%%%%%%%%%%%%%%%%%%%%%%%%%%%%%%%%%%%%%%%%%%%%%%%%%%%%%%%%%%%%%%%%%%%%%%%%%%%%%%%%%%%%%%%%%%%%%%%%%%%%%%%%%%%%%%%%%%%%%%%%%%%%%%%%%%%%%%%%%%%%%%%%%%%%%%%%

\section{Introduction}

\blfootnote{ \color{black}{$\ast$ This work was supported by Israel Science Foundation (grant no. 446/24), Chana and Heinrich Manderman
Chair on System Control at Tel Aviv University,}}
\blfootnote{ \color{black}{
		 a - Department of Electrical Engineering, Tel-Aviv University, Israel, (A. Jbara) adamjbara@tauex.tau.ac.il, (E. Fridman ) emilia@tauex.tau.ac.il.} } 
\blfootnote{ \color{black}{		 b - Center for Control Theory and Guidance Technology, Harbin Institute of Technology, Harbin, China, (X. Yang)   yangxuefei@hit.edu.cn }}
\blfootnote{ \color{black}{		$\dagger$ -  Corresponding Author}}

Extremum Seeking (ES) is a model-free adaptive control method that optimizes an unknown nonlinear output map in real time by tuning parameters to track its optimum \cite{ariyur2003real}. The first rigorous stability analysis of ES system  using averaging theory %and singular perturbations
was presented in the continuous-time \cite{krstic2000stability} and later extended to discrete-time \cite{choi2002extremum}. Subsequently, a wide range of ES theoretical advancements and diverse practical applications have emerged, including semi-global and global ES control \cite{discrete-LQ-control-linear-extremum-seeking-Krstic,tan2006non, tan2009global}, time-varying ES control \cite{guay2015time}, and ES for time-delay and PDE systems \cite{malisoff2021multivariable, oliveira2020multivariable,  oliveira2016extremum}. For a comprehensive overview see survey \cite{SCHEINKER-survey100ES}. A significant limitation of classical ES algorithms are their inherent bias, which provides only practical stability. To overcome this critical drawback, an unbiased ES algorithms has been  introduced in \cite{yilmaz2023exponential}.

The existing results on ES with delays are confined to known  constant or slowly-varying (with the delay derivative less than $1$) delays that may be large \cite{espitia2025prescribed,malisoff2021multivariable,oliveira2020multivariable, oliveira2022extremum}, whereas time-varying delay uncertainties are small \cite{Jbara25SCLES, mazenc2024bounded,Zhang23LieAvg}. ES in the presence of {\it unknown large time-varying delays} have not been studied yet, which motivates our present paper. %Most of the existing results on ES are qualitative proving convergence for high enough frequencies of the dithers.

%All of the above are qualitative results applicable to static maps, under the condition that the dither signals are sufficiently fast. However, they do not offer any quantitative bounds for the ES controller parameters. Furthermore, existing ES methods has only considered constant and known delays till now \cite{oliveira2022extremum, oliveira2016extremum, suttner2024overcoming}.

The first constructive results for classical and bounded ES of static quadratic maps were presented in  \cite{ES-Xuefei-CDC23, yang2023time, Zhang23LieAvg, zhu2022extremum} via
 the time-delay approach to averaging \cite{fridman2020averaging11},
 providing new qualitative and quantitative  (bounds on dither frequencies under some knowledge on the maps) results. The time-delay
approach has been extended to sampled-data ES \cite{zhu2022sampled}, discrete-time  noisy measurements \cite{yang2026constructive} %, discrete-time ES \cite{yang2023robust}
and non-quadratic maps \cite{pan2024extremum}.

Recently we have suggested a constructive approach to continuous-time unbiased ES \cite{Jbara25SCLES} in the presence of  measurement delays with a known large constant part and small time-varying uncertainty via a delay-free transformation.
%In addition, our method improved quantitative results compared to existing classical ES techniques \cite{yang2023robust, zhu2022extremum}.
For bounded ES of static quadratic maps with small uncertain measurement delays, a delay-free transformation and strict Lyapunov function were employed in \cite{mazenc2024bounded}.

In this paper, we introduce the discrete-time unbiased extremum seeking (ES) algorithm for $n$-dimensional ($n$D) static quadratic maps in the presence of unknown time-varying measurement delays bounded by known constants which can be large.
We derive new qualitative and quantitative results based on the delay-free transformation introduced  in \cite{JbaraAveragingIEEE2024}.
This is the first ES algorithm, which is robust with respect to {\it unknown large time-varying delays}.
 %and achieve unbiased exponential
%convergence.
We manage with such delays by choosing the dither   frequencies of the order of $\sqrt{\varepsilon}$, where the small parameter $\varepsilon>0$ appears in the dynamics of the real-time estimator. As expected, larger delays lead to a slower unbiased exponential convergence.
For the quantitative bounds on the controller parameters, we assume that the uncertain map Hessian lies within a known range that may be large. Differently from its continuous-time counterpart, the small parameter in the discrete-time case defines  the convergence rate, making a quantitative bound on this parameter particularly important.
We present also constructive conditions for the practical stability of the classical ES system. 

Our results are semi-global for globally quadratic maps, while for locally quadratic static maps, we provide a bound on the region of convergence. Our analysis shows that appropriate ES parameters can be found for any large unknown time-varying bounded delay. A numerical example highlights the  efficiency of the method.
A preliminary conference version of the paper in the case of known constant delay was presented in \cite{jbara2025discrete}.

\textbf{Notation:}  The notation used in this paper is fairly standard. $ \mathbb{N}$ refers to the set of positive integers.  $\mathbb{R}^n$ denotes the $n$-dimensional Euclidean space with vector norm $|\cdot|$, $\mathbb{R}^{n \times m}$ is the set of all $n \times m$ real matrices with the induced matrix norm $||\cdot||$. %, $0_n$ and $I_n$ are the zero matrix and the identity matrix of order $n$, respectively.
The notation $e_i \in \mathbb{R}^n$, ($i=1,2,...,n$) denotes the column vector with a $1$ in the $i$th coordinate and $0$'s elsewhere. % {\color{blue} The notation $E_{i,j} \in \mathbb{R}^{n \times n}$, ($i,j=1,2,...,n$) denotes the matrix with a $1$ in the $(i,j)$th coordinate and $0$'s elsewhere.}
The notation $P > 0$ for $P \in \mathbb{R}^{n \times n}$ means that $P$ is symmetric and positive definite. The superscript $T$ denotes matrix transposition. For $0 < P \in \mathbb{R}^{n \times n}$ and $x \in \mathbb{R}^n$, we write $|x|^2_P = x^TPx$.

%%%%%%%%%%%%%%%%%%%%%%%%%%%%%%%%%%%%%%%%%%%%%%%%%%%%%%%%%%%%%%%%%%%%%%%%%%%%%%%%%%%%%%%%%%%%%%%%%%%%%%%%%%%%%%%%%%%%%%%%%%%%%%%%%%%%%%%%%%%%%%%%%%%%%%%%%%%%%%%%%%%%%%%%%%%%%%%%%%%%%%%%%%%%%%%%%%%%

\section{Unbiased ES with unknown  delay}

Following \cite{oliveira2022extremum, oliveira2016extremum}, consider a multi-variable quadratic map
\begin{equation} \label{Eq-Quadratic-Form}
 Q (\theta(j)) = Q^*+ \frac{1}{2}|\theta(j)-\theta^*|_H^2, \quad  j \in \mathbb{Z}_+,
\end{equation}
where $\theta(j) \in \mathbb{R}^n$ is the vector input, $\theta^* \in \mathbb{R}^n$ and
$Q^* \in \mathbb{R}$ are unknown extremum point and extremum value, and $H$ is an unknown Hessian matrix. Without loss of generality, we consider a minimum seeking with $H > 0$, where \eqref{Eq-Quadratic-Form} has a minimum value $Q(j) = Q^*$ at $\theta = \theta^*$.

The delayed measurements of the map \eqref{Eq-Quadratic-Form} are given by
\begin{equation} \label{Eq-y-Q}
y(j)=\begin{cases}
0, \quad &  0 \leq j < D_M, \\
Q(\theta(j-D(j))), \quad & \quad \,\,\,\,\,  j \geq D_M,
\end{cases}
\end{equation}
where $D(j)$ is an unknown and bounded time-varying delay
\begin{equation}\label{Eq-Delay-Cond}
 0 \le D(j) \le D_M, \quad j \geq 0,
\end{equation}
where  $D_M \in \mathbb{Z}_{+} $ is the known bound.

We propose a discrete-time unbiased ES algorithm inspired by the continuous unbiased ES algorithm \cite{yilmaz2023exponential}. By using the measurements only, our algorithm constructs an input $\theta(j)$ that exponentially converges to $\theta^*$ in the presence of large unknown time-varying delay.
For globally quadratic maps, our results are semi-global, whereas for locally quadratic static maps, we provide a bound for the region of convergence (see Remark \ref{Rem-Non-Quad}).

%\subsection{Unbiased ES in the presence of unknown time-varying bounded delays}

Define the perturbation and demodulation vector
functions for $j \geq 0$ as
\begin{equation}\label{Eq-S(t)-M(t)}
\!\!\!\!\!\!\!\!\!
\begin{array}{ll}
&  S(j)=[S_1(j),...,S_n(j)]^T, \quad S_i(j) = a_i \sin(\omega_i j),\\
&  M(j)=[M_1(j),...,M_n(j)]^T, \quad M_i(j)=\frac{2}{a_i}\sin(\omega_i j),
 \end{array}
\end{equation}
where amplitudes $a_i$ are non-zero real numbers and the frequencies have a form
\begin{equation}\label{Eq-Omega}
\begin{array}{ll}
& \omega_i=\frac{2 \pi i}{T}, \,\, i=1,...,n, \,\, \varepsilon >0, \\
& T = \max \{ n D_M  \lfloor \frac{ 1 }{ \sqrt{\varepsilon} } \rfloor, 2n+1\}.
   \end{array}
\end{equation}

\subsection{Unbiased ES Algorithm:}  Define the exponentially decaying function
\begin{equation}\label{Eq-alpha}
\alpha(j) =  \alpha_0 \bar{\lambda}^{j},  \quad  \bar{\lambda}=1-\varepsilon \lambda,   \quad j \in \mathbb{Z}_+,
\end{equation}
where $\alpha_0$ and $\lambda$ are tuning positive parameters.

Define the input $\theta(j)$ with a real-time estimate $\hat{\theta}(j)$ of $\theta^*$  as follows:
\begin{equation}  \label{Eq-hattheta-Un}
\theta(j) = \begin{cases}
  \theta (0), \qquad \qquad \qquad  \quad  0 \leq  j \leq D_{M} , \\
 \hat{\theta}(j) +  \alpha(j) S(j), \qquad \quad \,\,\,\, j > D_{M},
 \end{cases}
\end{equation}
with $S(j)$ and $\alpha(j)$ given by \eqref{Eq-S(t)-M(t)} and \eqref{Eq-alpha}, respectively.
Choose a scalar adaptation gain $k>0$ and a small tuning parameter $\varepsilon>0$. As illustrated in Fig. 1, the unbiased ES algorithm has the form
\begin{equation}\label{Eq-dervative-hattheta-Un-1}
\!\!\! \!\!\! \!\!\!
\begin{array}{ll}
&  \hat{\theta}(j+1) =    \hat{\theta}(j) - \varepsilon  \frac{kM(j)}{\alpha(j)} [y(j)-\eta(j)], \quad j \geq D_M, \\
&   \hat{\theta}(j) = \theta(0), \qquad  \qquad \qquad \qquad \qquad \,\, 0 \leq j \leq D_M,
\end{array}
\end{equation}
where the high-pass filter state $\eta(j)$ is governed  by
\begin{equation}\label{Eq-dervative-eta-Un}
\!\!
\begin{array}{ll}
& \eta(j+1) = \eta(j) + \varepsilon \omega_h [y(j) - \eta(j)],   \qquad    j \geq D_M,\\
& \eta(j)=Q_0,  \qquad \qquad \qquad \qquad \qquad \qquad \,\,    j = D_M,
\end{array}
\end{equation}
with some $\omega_h>0$.
%We will further present \eqref{Eq-dervative-hattheta-Un-1} as
%\begin{equation}\label{Eq-dervative-hattheta-Un-2}
%\!\!\!\!\!\!\!
%\begin{array}{ll}
 % \hat{\theta}(j+1)  & =   \hat{\theta}(j)   - \varepsilon  \frac{kM(j)}{\alpha(j)} [y(j)-\eta(j)] \\
%& = \hat{\theta}(j)   - \varepsilon  \frac{kM(j)}{\alpha(j)} [Q(\theta(j-D(j)))-\eta(j)] \\
%\hat{\theta}(j+1) \!\!\!\!\!\! & = \hat{\theta}(j)   - \varepsilon  \frac{kM(j)}{\alpha(j)} [Q(\theta(j))-\eta(j)] \\
%& \quad  - \varepsilon  \frac{kM(j)}{\alpha(j)} [Q(\theta(j-D(j)))-Q(\theta(j))].
%\end{array}
%\end{equation}
%for $j \geq D_M$.

\begin{figure}[h]
\includegraphics[width=7.5cm, height=6cm]{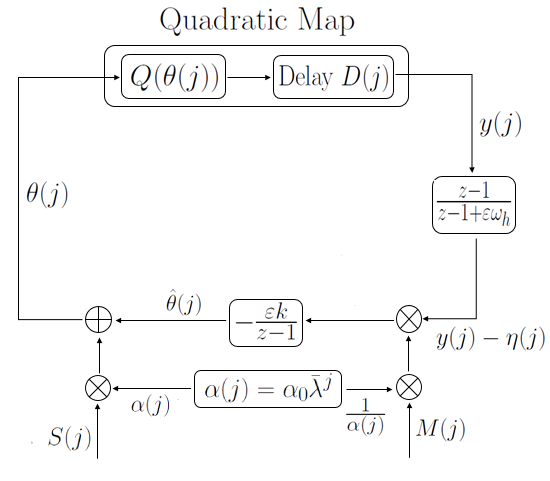}
\centering
\caption{
 Unbiased ES algorithm in the presence of time-varying measurement delay $D(j)$.
}\label{Fig-1}
\end{figure}
%Differently from the continuous-time unbiased ES algorithm, the discrete-time algorithm depends on two small tuning parameters $\varepsilon$ (the integration step) and $T$ (which eventuality determines the frequencies $\omega_i= O(\sqrt \varepsilon)$) that allows for robustness with respect to unknown time-varying delays with large $D_M$.

%In contrast to its continuous-time counterpart \cite{Jbara25SCLES, yilmaz2023exponential} with one small parameter (defining high dither frequencies),
The proposed discrete-time unbiased  ES algorithm  utilizes two small tuning parameters: the  gain $\varepsilon$ in the dynamics of $\hat \theta$ and the   dither  frequencies $\omega_i = O(\sqrt{\varepsilon})$ (that are low, but the dithers are varying faster than $\hat \theta$). The latter choice of frequencies allows to preserve the stability in the presence of large unknown time-varying delays.

%Differently from the continuous-time unbiased ES algorithm, the discrete-time algorithm depends on two key parameters: the integration step $\varepsilon$ and the period $T$.

%The latter parameter $T$ defines the dither frequencies $\omega_i = O(\sqrt{\varepsilon})$, which facilitates robustness against unknown time-varying delays with significant upper-bound of $D_M$.
%These parameters allow for dither frequencies $\omega_i$ of the order $O(\sqrt{\varepsilon})$, enabling the algorithm to handle unknown time-varying delays with significant upper-bound of $D_M$.
%Our objective is to derive constructive, quantitative conditions for semi-global exponential convergence in the presence of time-varying delays. Specifically, for an unknown time-varying bounded delay $D(j)$  and a given initial state ball for $\theta(0)$, we seek to identify appropriate exponentially decaying gains, high-pass filters, and perturbation frequencies. The results below provide simple scalar inequalities that guarantee the exponential stability of the estimation error and provide quantitative bounds on the controller parameters.

Similar to \cite{yilmaz2023exponential}, given any desirable $\lambda >0$ (associated with decay rate), we choose positive tuning parameters $k$ and $\omega_h$  that satisfy the following inequalities:
\begin{equation} \label{Ineq-Assumption-2}
k  H   > \lambda I, \quad \omega_h > 2 \lambda.
\end{equation}
Specifically, the inequalities in \eqref{Ineq-Assumption-2} dictate that the adaptation (learning) rates need to be greater than $\lambda$ associated with the perturbation (exploration) signal.  In addition, it is worth noting that while inequalities \eqref{Ineq-Assumption-2} do not dependent of $D_M$, our conditions  in
Theorem \ref{THM-1} below (see inequalities \eqref{Ineq-THM-Un-condition-1}-\eqref{Ineq-THM-Un-condition-3}) imply that a larger
$D_M$ leads  to a slower convergence (larger decay rate $\bar{\lambda}$).

\begin{remark}\label{Rem-Diff-UnB-Class-1}
As in the continuoius case \cite{Jbara25SCLES, yilmaz2023exponential}, the discrete-time unbiased ES algorithm  \eqref{Eq-dervative-hattheta-Un-1}, \eqref{Eq-dervative-eta-Un} is equipped with high-pass filter, and exponentially decaying perturbation and growing demodulation signals. Despite the exponentially growing signal in the algorithm, the high-pass filter has a crucial role in guaranteeing the unbiased convergence  by ensuring the exponential decay of $y(j)-\eta(j)$ to zero at the rate of $\bar{\lambda}^2$ (see \eqref{Ineq-y-Q*-delay} below).
\end{remark}
Define the estimation errors $\tilde{\theta}(j)$ and $\tilde{\eta}(j)$ as
\begin{equation}\label{Eq-tildetheta-tildeeta-Un}
	\begin{array}{ll}
  \tilde{\theta}(j) = \hat{\theta}(j)-{\theta}^*,\quad & j \geq 0, \\
    \tilde{\eta}(j) = \eta(j)-Q^*, \quad  & j \geq D_M.
    \end{array}
\end{equation}
Then, the estimation errors are governed by
\begin{equation}\label{Eq-dervative-tildetheta-Un-1}
\!\!\!\!\!\!\!\!
\begin{array}{ll}
&    \tilde{\theta}(j+1) = \tilde{\theta}(j)      \\
& \qquad   - \varepsilon  \frac{kM(j)}{\alpha(j)} \Big[ y(j)-Q^*-\frac{1}{2}|\tilde{\theta}(j)+\alpha(j) S(j)|_H^2 \Big]  \\
&   \,\,  - \varepsilon  \frac{kM(j)}{\alpha(j)} \left[ \frac{1}{2}|\tilde{\theta}(j)+\alpha(j) S(j)|_H^2 -\tilde{\eta}(j)\right],  \,\,\,\,  j \geq  D_M, \\
& \tilde \theta(j)=\theta(0)-\theta^*,  \qquad \qquad  0 \leq j \leq D_M,
\end{array}
\end{equation}
and
\begin{equation}\label{Eq-dervative-tildeeta-Un}
\!\!\!\!\!\!\!\!\!
\begin{array}{ll}
 & \tilde{\eta}(j+1) =  \tilde{\eta}(j)   - \varepsilon \omega_h \Big[ \tilde{\eta}(j)     - y(j) +Q^* \Big],   \,\,\,  j \geq D_M,\\
& \tilde{\eta}(j)= Q_0 - Q^*,  \qquad \qquad  \qquad\qquad \qquad  \quad \,\,    j = D_M.
\end{array}
\end{equation}
For small enough frequencies $\omega_i$ (large $T$), we have
\begin{equation}\label{Eq-approx-sin-cos-eps}
\begin{array}{ll}
\qquad \sin(\frac{2\pi i}{T}) \!\!\!\!\!\! & \approx  \frac{2\pi i}{T} =O(\sqrt{\varepsilon}), \\
1-\cos(\frac{2 \pi i}{T}) \!\!\!\!\!\!  & \approx  \left( \frac{2\pi i}{T} \right)^2  =O(\varepsilon), \quad i=1,2,...,n,
\end{array}
\end{equation}
which lead to
 \begin{equation}\label{Eq-diff-alpha-S}
 \!\!\!\!\!\!\!\!\!\!\!\!
\begin{array}{ll}
%{\color{red} \alpha(j) S(j) -  \alpha(j-1)S(j-1) } & & =  [\alpha_0 \bar{\lambda}^j] [a_i\sin(\omega_i j)] -  [\alpha_0 \bar{\lambda}^{j-1}] [a_i\sin(\omega_i (j-1))]\\
%& = a_i \alpha_0 \bar{\lambda}^{j-1} (1- \varepsilon \lambda) \sin(\frac{2 \pi i}{T} j)\\
%& \quad - a_i \alpha_0 \bar{\lambda}^{j-1}  [\sin(\frac{2 \pi i}{T} j) \cos(\frac{2 \pi i}{T}) -\cos (\frac{2 \pi i}{T} j) \sin(\frac{2 \pi i }{T}) ]  \\
%& = a_i \alpha_0 \bar{\lambda}^{j-1}   \sin(\frac{2 \pi i}{T} j) - \varepsilon \lambda  a \alpha_0 \bar{\lambda}^{j-1}    \sin(\frac{2 \pi i}{T} j)\\
%& \quad - a_i \alpha_0 \bar{\lambda}^{j-1}  [\sin(\frac{2 \pi i}{T} j) \cos(\frac{2 \pi i}{T}) -\cos (\frac{2 \pi i}{T} j) \sin(\frac{2 \pi i}{T}) ]  \\
%& = - \varepsilon  a_i \alpha_0 \lambda   \bar{\lambda}^{j-1}    \sin(\frac{2 \pi i}{T} j)  \\
%& \quad + a_i \alpha_0 \bar{\lambda}^{j-1}   \sin(\frac{2 \pi i}{T} j) - a_i \alpha_0 \bar{\lambda}^{j-1}   \sin(\frac{2 \pi i}{T} j) \cos(\frac{2 \pi i }{T})   \\
%& \quad + a_i \alpha_0 \bar{\lambda}^{j-1}   \cos(\frac{2 \pi i}{T} j)\sin(\frac{2 \pi i}{T})   \\
%&   =  -  \varepsilon   a_i \alpha_0 \lambda   \bar{\lambda}^{j-1}    \sin(\frac{2 \pi i}{T} j)     \\
%& \quad   + a_i  \alpha_0 \bar{\lambda}^{j-1}    \underbrace{[1-  \cos(\frac{2 \pi  i }{T})]}_{O(\varepsilon)}   \sin(\frac{2 \pi i}{T} j)   \\
%& \quad  + a_i \alpha_0 \bar{\lambda}^{j-1}   \underbrace{\sin(\frac{2 \pi i }{T}) }_{O(\sqrt{\varepsilon})}  \cos(\frac{2 \pi i}{T} j)\\
& \alpha(j) S_i(j) -  \alpha(j-1)S_i(j-1)\\
& \qquad  =  -  \varepsilon   a_i \alpha_0 \lambda   \bar{\lambda}^{j-1}    \sin(\frac{2 \pi i}{T} j)     \\
& \qquad \quad  + a_i \alpha_0 \bar{\lambda}^{j-1}    \underbrace{[1-  \cos(\frac{2 \pi  i }{T})]}_{O(\varepsilon)}   \sin(\frac{2 \pi i}{T} j)   \\
& \qquad \quad + a_i \alpha_0 \bar{\lambda}^{j-1}   \underbrace{\sin(\frac{2 \pi i }{T}) }_{O(\sqrt{\varepsilon})}  \cos(\frac{2 \pi i}{T} j).
\end{array}
\end{equation}
Thus, we obtain
 \begin{equation}\label{Eq-diff-Q}
  \!\!\!\!\! \!\!\!\!\!
\begin{array}{ll}
% y(j)-Q^*-\frac{1}{2}|\tilde{\theta}(j)+\alpha(j) S(j)|_H^2    & = \Big[Q^* + \frac{1}{2}|\tilde{\theta}(j)+\alpha(j-D(j)) S(j-D(j))|_H^2 \Big] -Q^*-\frac{1}{2}|\tilde{\theta}(j)+\alpha(j) S(j)|_H^2 \\
% & =   \frac{1}{2}|\tilde{\theta}(j)+\alpha(j-D(j)) S(j-D(j))|_H^2  - \frac{1}{2}|\tilde{\theta}(j)+\alpha(j) S(j)|_H^2 \\
% &= \frac{1}{2} \Big[  \tilde{\theta}(j-D(j)) + \tilde{\theta}(j)   + \alpha(j-D(j))S(j-D(j))  + \alpha(j) S(j)   \Big]^T  \times \\
%&   H \Big[ \sum_{k=j-D(j)+1}^{j}\left(     \underbrace{\tilde{\theta}(i-1) - \tilde{\theta}(i)}_{O(\varepsilon)} \right)  \\
%&  + \sum_{k=j-D(j)+1}^{j} \left(    \underbrace{ \alpha(j-1)S(j-1)  - \alpha(j)S(j)  }_{O(\sqrt{\varepsilon})}   \right) \Big]   = O(\sqrt{\varepsilon}).\\
& y(j)-Q^*-\frac{1}{2}|\tilde{\theta}(j)+\alpha(j) S(j)|_H^2   \\
& \quad   = \frac{1}{2} \Big[  \tilde{\theta}(j-D(j)) + \tilde{\theta}(j) \\
&  \qquad + \alpha(j-D(j))S(j-D(j))  + \alpha(j) S(j)   \Big]^T H  \\
&  \times \Big[ \sum_{k=j-D(j)+1}^{j}\left(     \underbrace{\tilde{\theta}(k-1) - \tilde{\theta}(k)}_{O(\varepsilon)} \right)  \\
&  + \sum_{k=j-D(j)+1}^{j} \left(    \underbrace{ \alpha(k-1)S(k-1)  - \alpha(k)S(k)  }_{O(\sqrt{\varepsilon})}   \right) \Big]  \\
&  \qquad \qquad \qquad \qquad \qquad \qquad = O(\sqrt{\varepsilon}).
\end{array}
\end{equation}
Equation \eqref{Eq-dervative-tildetheta-Un-1} can be further expressed as
\begin{equation}\label{Eq-dervative-tildetheta-Un-2}
\!\!\!\!
\begin{array}{ll}
&  \tilde{\theta}(j+1)  = \tilde{\theta}(j)          - \varepsilon  \frac{kM(j)}{\alpha(j)}  \Big[ \frac{1}{2}|\tilde{\theta}(j)+\alpha(j) S(j)|_H^2  \\
& \qquad \qquad \qquad \qquad    -\tilde{\eta}(j)\Big] + \varepsilon    \bar{\Delta}_{\sqrt{\varepsilon}}(j),  \,\,\,\,  j \geq  D_M,
\end{array}
\end{equation}
with
%$O(\sqrt{\varepsilon})$ term %(cf. \eqref{Eq-diff-Q})
%given by $\bar{\Delta}_{\sqrt{\varepsilon}}(j)$
\begin{equation}\label{Eq-Delta-eps}
\!\!\!\!
\begin{array}{ll}
& \bar{\Delta}_{\sqrt{\varepsilon}}(j)   = -     \frac{kM(j)}{\alpha(j)}   \Big[y(j) - Q^*  \\
&  \qquad \qquad    -\frac{1}{2}|\tilde{\theta}(j)+\alpha(j) S(j)|_H^2 \Big]=O(\sqrt{\varepsilon}).
 \end{array}
\end{equation}
%Also, it follows from \eqref{Eq-diff-Q} that $ \bar{\Delta}_{\sqrt{\varepsilon}}(j) = O(\sqrt{\varepsilon})$.
One can easily show that
\begin{equation}\label{Eq-integral-matrices}
\begin{array}{ll}
 & \frac{1}{T} \sum_{j=t}^{t+T-1}  M(j)  =0,\\ %= col \{ \frac{1}{\varepsilon} \int_{t}^{t+ \varepsilon} \sin(\frac{2\pi i s}{\varepsilon}) ds \} = 0, \\
 & \frac{1}{T} \sum_{j=t}^{t+T-1} M(j)S^T(j)   = I, \\
 & \frac{1}{T} \sum_{j=t}^{t+T-1}  M(j) S^T(j) HS(j)  = 0. 		
\end{array}
\end{equation}
Thus, the nominal (averaged) system corresponding to \eqref{Eq-dervative-tildetheta-Un-2} for small enough $\varepsilon$ is given by
\begin{equation}\label{Eq-Averaged-Sys}
\tilde{\theta}_{av}(j+1) = (1 - \varepsilon k H)   \tilde{\theta}_{av}(j),
\end{equation}
which is exponentially stable. Differently from the continuous-time case, the decay rate of \eqref{Eq-Averaged-Sys} depends on $\varepsilon$. Therefore, bounds on $\varepsilon$ for the estimation error system that we aim to provide are important for achieving a desirable decay rate.

\subsection{System transformation and main results}
We will perform averaging-based analysis
of \eqref{Eq-dervative-tildetheta-Un-3} via a delay-free transformation introduced in \cite{JbaraAveragingIEEE2024}.
With notations
\begin{equation} \label{Eq-a_i}
\!\!\!\!\!\!
\begin{array}{ll}
& \!\!\!\!  a_{1}(j) =   - k M(j)S^T(j) H + k H,  \ a_{2}(j) =   k M(j),    \\
& \!\!\!\! 	a_{3}(j) =  -\frac{1}{2} k M(j) S^T(j)H  S(j), \ a_{4}(j) =  -\frac{1}{2} k M(j),
\end{array}
\end{equation}
equation \eqref{Eq-dervative-tildetheta-Un-2} for $j \geq  D_M$ can be  presented as
\begin{equation}\label{Eq-dervative-tildetheta-Un-3}
\!\!\!\!\!\!\! \!\!
\begin{array}{ll}
%&  \tilde{\theta}(j+1)  = \tilde{\theta}(j)   - \varepsilon    k M(j)\left[ Q^*+   \frac{H}{2} |\tilde{\theta}(j)  + S(j)|^2  \right] +{  \bar{\Delta}_{\sqrt{\varepsilon}}(j) } \\
%& \qquad \quad \,\,\,\, =   \tilde{\theta}(j)    - \varepsilon    \frac{ k M(j)}{2}  |\tilde{\theta}(j)|^2_{ H }  - \varepsilon    \frac{ k M(j)}{2}  |S(j)|^2_{ H }   \\
%& \qquad \qquad \quad  - \varepsilon   k M(j)S^T(j) H \tilde{\theta}(j)  - \varepsilon    k M(j) Q^* +{ \bar{\Delta}_{\sqrt{\varepsilon}}(j) } \\
 & \tilde{\theta}(j+1)   =  [1-\varepsilon k H] \tilde{\theta}(j) + \varepsilon a_1(j) \tilde{\theta}(j)  \\
 & \qquad \qquad \,\,\, +  \varepsilon   a_2(j) \alpha^{-1}(j) \tilde{\eta}(j)   + \varepsilon  a_3(j)\alpha (j)   \\
 &  \qquad \qquad \,\,\, + \varepsilon   a_4(j) \alpha^{-1}(j)  |\tilde{\theta}(j)|^2_{ H } + \varepsilon    \bar{\Delta}_{\sqrt{\varepsilon}}(j),
\end{array}
\end{equation}
 with $\tilde{\theta}(j) =  \theta(0)-\theta^*$ for $0 \leq j \leq D_M$.

Define the functions
\begin{equation} \label{Eq-rho}
  \rho_{l}(j) := -\frac{\varepsilon}{T} \sum_{i=j}^{j+ T-1} (j+T-i)  a_l(i), \quad l=1,2,3,4.
\end{equation}
Due to \eqref{Eq-integral-matrices} and \eqref{Eq-a_i}, we have 
\begin{equation}
\begin{array}{ll}
 \rho_{l}(j+1) -\rho_{l}(j) = \varepsilon a_{l}(j), \quad j \geq 0.
\end{array}
\end{equation}
%{\color{red} Recall that $T$ depends on $\varepsilon$ ( it is $O(\frac{1}{  \sqrt{\varepsilon}  } )$).}
Since $a_l(j)$, $l=1,2,3,4$ are bounded with the bound (denoted by $\bar{a}_l$) that does not depend on $\varepsilon$, then we have
\begin{equation}\label{Eq-rho-gamma-bounds-eps}
\begin{array}{ll}
||\rho_{1}(j)||  \!\!\!\!\! & \leq   \frac{\varepsilon}{T} \sum_{i=j}^{j+ T-1} (j+T-i)\bar{a}_1  \\
&= \frac{\varepsilon \bar{a}_l}{T} \left[ \frac{T(T+1)}{2} \right]  \leq   \varepsilon \bar{a}_1 T    = O( \sqrt{\varepsilon} ),
\end{array}
\end{equation}
and $|\rho_{l}(j)| \leq \varepsilon \bar{a}_l T = O(\sqrt{\varepsilon})$, $l=2,3,4$.
%Also, we get
%\begin{equation*}
%\begin{array}{ll}
%  \rho_{l}(j+1) & = -\frac{\varepsilon}{T} \sum_{i=j+1}^{j+ T} (j+1+T-i)  a_l(i)  + \frac{\varepsilon}{T} \sum_{i=j}^{j+ T-1} (j+T-i)  a_l(i)  \\
% & = -\frac{\varepsilon}{T} \sum_{i=j+1}^{j+ T} (j+T-i)  a_l(i)  - {\color{red} \underbrace{  \frac{\varepsilon}{T} \sum_{i=j+1}^{j+ T}    a_l(i)}_{=0}} \\
% & \qquad \qquad \qquad  + \frac{\varepsilon}{T} \sum_{i=j}^{j+ T-1} (j+T-i)  a_l(i)  \\
% & = -\frac{\varepsilon}{T} \sum_{i=j+1}^{j+ T-1} (j+T-i)  a_l(i)    \\
% & \qquad \qquad \qquad  + \frac{\varepsilon}{T} \sum_{i=j+1}^{j+ T-1} (j+T-i)  a_l(i) + \frac{\varepsilon}{T}  T  a_l(j) \\
% & =   \varepsilon   a_l(j)
%  \end{array}
%\end{equation*}
%Let $\bar{\rho}_{j}$ be upper bounds of $\rho_{l}(j)$:
%\begin{equation} \label{Eq-rho-gamma-bounds}
%\begin{array}{ll}
%||\rho_{l}(j)|| \leq \sqrt{\varepsilon}   \bar{\rho}_{l}, \quad  \rho_{l}(j+1) -\rho_{l}(j) = \varepsilon a_{l}(j),
%\end{array}
%\end{equation}
%with
%{\color{red}\begin{equation}\label{Eq-bound-rho}
%\begin{array}{ll}
%& \bar{\rho}_{1} = 0.192 k H_M D_M, \\
%& \bar{\rho}_{2} = 2\bar{\rho}_{4} =0.3848  \frac{k}{a}D_M,\\
%& \bar{\rho}_{3} = 0.136 ak H_M D_M.
%\end{array}
%\end{equation}}
Define the function
\begin{equation} \label{Eq-G-delay}
 G(j) =
 \begin{cases}
  0, \quad \hspace{2.5 cm}  0 \leq j \leq D_M-1, \\
 \!\!\! \begin{array}{ll}
\rho_{1}(j)  \tilde{\theta}(j )  + \rho_{2}(j) \alpha^{-1}(j) \tilde{\eta}(j) + \rho_{3}(j) \alpha(j)  \\
     + \rho_{4}(j) \alpha^{-1}(j)  |\tilde{\theta}(j)|_{ H }^2, \quad \,\,\, j \geq D_M.
 \end{array}
 \end{cases}
\end{equation}
Note that the term $G(j)$ is $O(\sqrt{\varepsilon})$, provided $\tilde{\theta}(j)$ is $O(1)$. We have
\begin{equation}\label{Eq-derivative-G}
\!\!\!\!\!\!\!\!
\begin{array}{ll}
& G(j+1) - G(j)  =      \varepsilon a_{1}(j)  \tilde{\theta}(j)   + \varepsilon a_2(j)  \alpha^{-1}(j)  \tilde{\eta}(j)  \\
& \qquad \qquad \qquad + \varepsilon a_3(j) \alpha(j) + \varepsilon a_4(j) \alpha^{-1}(j)  |\tilde{\theta}(j)|_{ H }^2    \\
& \qquad \qquad \qquad  +  \varepsilon \rho_{1}(j+1) [\Delta(j)+ \bar{\Delta}_{\sqrt{\varepsilon}}(j) ] \\
& \qquad  + \varepsilon \omega_h  \rho_{2}(j+1) \alpha^{-1}(j)  \Big[ \frac{1}{2}|\tilde{\theta}(j-D(j))\\
& \qquad \qquad \qquad +\alpha(j-D(j)) S(j-D(j))|_H^2  - \tilde{\eta}(j) \Big] \\
&  + \rho_{2}(j+1) \alpha^{-1}(j) ( \bar{\lambda}^{-1}-1)   \Big( \tilde{\eta}(j)[1-\varepsilon \omega_h] \\
& \quad   +   \frac{ \varepsilon \omega_h}{2}|\tilde{\theta}(j-D(j)) +\alpha(j-D(j)) S(j-D(j))|_H^2   \Big) \\
&   \qquad \qquad \qquad  + \rho_{3}(j+1) \alpha(j) [\bar{\lambda}-1]    \\
&\qquad  \quad + 2\varepsilon \rho_{4}(j+1)    \alpha^{-1}(j)  \tilde{\theta}^T(j) H [\Delta(j)+ \bar{\Delta}_{\sqrt{\varepsilon}}(j) ] \\
& \qquad \qquad  + \varepsilon^2 \rho_{4}(j+1) \alpha^{-1}(j)   |   \Delta(j)+ \bar{\Delta}_{\sqrt{\varepsilon}}(j) |_{ H }^2 \\
& \qquad \qquad  + \rho_{4}(j+1) \alpha^{-1}(j) (\bar{\lambda}^{-1}-1) \\
& \qquad \qquad  \qquad\qquad \times |\tilde{\theta}(j) + \varepsilon (\Delta(j)+ \bar{\Delta}_{\sqrt{\varepsilon}}(j)  )|_{ H }^2,
   \end{array}
\end{equation}
for $j \geq D_M$ with
\begin{equation}\label{Eq-Delta}
\!\!\!\!\! \!\!\!\!\!\!\!\!\!
\begin{array}{ll}
  & \Delta(j) =  [a_1(j)-   kH]   \tilde{\theta}(j)    +  a_2(j)  \alpha^{-1}(j)  \tilde{\eta}(j) \\
& +   a_3(j) \alpha(j) +   a_4(j) \alpha^{-1}(j)  |\tilde{\theta}(j)|_{ H }^2, \quad   j \geq D_M.
  \end{array}
\end{equation}
Consider the following transformation
\begin{equation} \label{Eq-z-delay}
	z(j) =  \tilde{\theta}(j) - G(j), \quad   j \geq D_M.
\end{equation}
By employing \eqref{Eq-dervative-tildetheta-Un-3} and \eqref{Eq-z-delay}, we obtain
\begin{equation}\label{Eq-dervative-z-Y-delay}
\!\!\!\!\!\!\!\!\!
\begin{array}{ll}
  & {z}(j+1) - z(j)=  - \varepsilon kH z(j) + \underbrace{Y(j)}_{O(\varepsilon^{\frac{3}{2}})}, \quad j \geq D_M,\\
%&   |z(j)| = |\tilde{\theta}(0)| \leq \sigma_0, \qquad \qquad     0 \leq  j \leq D_M,\\
&  z(D_M)  = \tilde{\theta}(D_M) - G(D_M),
   \end{array}
\end{equation}
where
\begin{equation}\label{Eq-Y-delay}
\!\!\!\!\!\!\!\!\!
\begin{array}{ll}
  &{Y(j)} =  -\varepsilon kH G(j) -  \varepsilon   \rho_{1}(j+1) (\Delta(j)+ \bar{\Delta}_{\sqrt{\varepsilon}}(j) ) \\
& - \varepsilon \omega_h  \rho_{2}(j+1) \alpha^{-1}(j)  \Big[ \frac{1}{2}|\tilde{\theta}(j-D(j))\\
& \qquad \qquad \quad +\alpha(j-D(j)) S(j-D(j))|_H^2  - \tilde{\eta}(j) \Big] \\
&  - \rho_{2}(j+1) \alpha^{-1}(j) ( \bar{\lambda}^{-1}-1)   \Big( \tilde{\eta}(j)[1-\varepsilon \omega_h] \\
& \quad   +   \frac{ \varepsilon \omega_h}{2}|\tilde{\theta}(j-D(j)) +\alpha(j-D(j)) S(j-D(j))|_H^2   \Big) \\
& \qquad   - \rho_{3}(j+1) \alpha(j)  [\bar{\lambda}-1] + \varepsilon   \bar{\Delta}_{\sqrt{\varepsilon}}(j)    \\
& - 2\varepsilon   \rho_{4}(j+1)    \alpha^{-1}(j)  \tilde{\theta}^T(j) H [\Delta(j)+ \bar{\Delta}_{\sqrt{\varepsilon}}(j) ] \\
&  - \varepsilon^2  \rho_{4}(j+1)    \alpha^{-1}(j)   |   \Delta(j)+ \bar{\Delta}_{\sqrt{\varepsilon}}(j)   |_{ H }^2 \\
&  -   \rho_{4}(j+1)   \alpha^{-1}(j)  (\bar{\lambda}^{-1}-1) \\
& \qquad   \times|\tilde{\theta}(j) + \varepsilon (\Delta(j)+ \bar{\Delta}_{\sqrt{\varepsilon}}(j)  )|_{ H }^2, \quad j \geq D_M.
\end{array}
\end{equation}
Since $G(j)$ is $O(\sqrt{\varepsilon})$, it follows that $Y(j)$ is $O(\varepsilon^{\frac{3}{2}})$, provided $\tilde{\theta}(j)$ is $O(1)$. %The terms $G(j)$ and $Y(j)$ are of the order of $O(\sqrt{\varepsilon})$ and $O(\sqrt{\varepsilon}^3)$ respectively, provided $\tilde{\theta}(j)$ is of the order of $O(1)$.

%{\color{red}
%The following proposition establishes a qualitative stability result through averaging analysis:
%\begin{proposition}\label{Prop-Qualitative}
%Given any delay bound $D_M$, and parameters $a_i$, $k$, $\lambda$, $\omega_h$ satisfying \eqref{Ineq-Assumption-2}, there exist sufficiently small tuning parameters   $\varepsilon$ and $\omega_i=O(\sqrt{\varepsilon})$, $i=1,...,n$ such that the closed-loop system in (\ref{Eq-dervative-hattheta-Un-1}) and (\ref{Eq-dervative-eta-Un}) is exponentially stable at the origin. The latter means that the input $\theta(j)$ and output $y(j)$ exponentially converge to $\theta^*$ and $Q^*$, respectively.
%\end{proposition}
%{\bf Proof:}
%The proof of Proposition \ref{Prop-Qualitative} is omitted as it is analogous to that of Theorem \ref{THM-1}.
%$\hfill \Box$
%}

While $H$ can be accurately estimated when delays are constant and known by using results of \cite{ghaffari2012multivariable} (see Remark 1 in \cite{Jbara25SCLES}), the presence of unknown time-varying delays makes obtaining a reliable estimation mathematically intractable. For quantitative results, we assume that uncertain $H$ lies within a known range:

%To derive rigorous quantitative results, it is necessary to assume that the unknown Hessian $H$ belongs to a known compact set. Although $H$ can be accurately estimated when delays are constant and known \cite{ghaffari2012multivariable}, the presence of unknown time-varying delays makes obtaining a reliable estimate mathematically intractable. Consequently, to facilitate the derivation of explicit stability bounds, we assume that $H$ is bounded as follows:

%Similar to the approach in \cite{ghaffari2012multivariable}, estimating the Hessian from the measurements $y(j)$ appears to be possible provided delay is constant and known \cite{JbaraESCDC25}. However, due to the presence of an unknown time-varying bounded delay, an efficient estimate of $H$ seems to be not possible. Note that all the existing results for known constant delays use such an estimate. Therefore, to have efficient quantitative bounds on the system parameters, we assume that uncertain $H$ lies within a known range:

\begin{assumption}\label{Assumption-Hessian}
The Hessian $H$ is uncertain and subject to $H_m I \leq H \leq H_M I $, where $H_m$ and $H_M$ are  known positive scalars.
\end{assumption}
%From Assumption \ref{Assumption-Hessian}, The following Lemma provides accurate upper-bounds on each function in \eqref{Eq-rho}:
Under Assumption \ref{Assumption-Hessian}, the following bounds are derived in the Appendix:
\begin{lem}\label{Lem-Upper-Bounds}
The functions $\rho_{i}(j)$ ($i=1,2,3,4$) defined by \eqref{Eq-rho} are bounded as follows for $t\ge 0$:
\begin{equation} \label{Eq-rho-gamma-bounds}
\!\!\!\!\!\!\!\!\!
\begin{array}{ll}
&||\rho_{1}(j)|| \leq  \bar{\rho}_{1}\sqrt{\varepsilon} := 0.19245   n D_M  k H_M \times  \\
 & \qquad \qquad  \qquad    \Big[ \sum_{i=1 }^n   \frac{1 }{ i } +   \sum_{i \neq l}^n  \frac{a_l}{a_i}  [   \frac{ 1 }{ |i-l| }   +   \frac{ 1 }{ i+l }  ]  \Big]  \sqrt{\varepsilon},\\
 &|\rho_{2}(j)| \leq   \bar{\rho}_{2} \sqrt{\varepsilon} :=   0.3849   n D_M  k \left[ \sum_{i=1 }^n   \frac{1 }{ ia_i }   \right]  \sqrt{\varepsilon},\\
&  |\rho_{3}(j)| \leq  \bar{\rho}_{3} \sqrt{\varepsilon} :=  k  H_M  \left( \sum_{s=1}^{n}a_s^2\right) \sqrt{\sum_{i=1}^n \frac{c_i^2}{a_i^2} }\sqrt{\varepsilon},  \\
 &|\rho_{4}(j)| \leq   \bar{\rho}_{4} \sqrt{\varepsilon} :=  0.19245  n D_M  k \left[ \sum_{i=1 }^n   \frac{1 }{ ia_i }   \right]  \sqrt{\varepsilon},
\end{array}
\end{equation}
where
\begin{equation*}
\begin{array}{ll}
& c_i =   \Big[  \frac{ ( n  D_M + \sqrt{\varepsilon})(2n  D_M    + \sqrt{\varepsilon})}{12}  \\
& \qquad \qquad + \frac{ \sqrt{\varepsilon} ( \sqrt{\varepsilon} + n D_M    )  + D_M^2 n^2( 1+   \frac{ 1 }{3.3072 \pi  i }   )  }{3.3072 \pi i }  \Big].
\end{array}
\end{equation*}

\end{lem}

In addition, we assume the following uncertainties regarding the extremum point $\theta^*$ (influencing the initial condition) and  output map $y(t)$:
\begin{assumption}\label{Assumption-theta}
 The extremum point $\theta^*$ to be sought is uncertain from a known ball $B$ with radius $\sigma_0$ where  its elements satisfy $\theta^*_i \in [\underline{\theta}^*_i, \overline{\theta}^*_i]$, $i=1,...,n$  with $\sum_{i=1}^n ( \overline{\theta}^*_i-\underline{\theta}^*_i)^2 = \sigma_0^2$. The extremum value $Q^*$ is uncertain subject to $|Q^*-Q_0| \leq  \Delta_Q$  with known $Q_0$ and $\Delta_Q>0$.
\end{assumption}
We are in a position to formulate our main result on semi-global exponential stability of the estimation error system \eqref{Eq-dervative-tildetheta-Un-1}, \eqref{Eq-dervative-tildeeta-Un}:

\begin{theorem}\label{THM-1}
Let Assumptions \ref{Assumption-Hessian} and \ref{Assumption-theta} hold. Given
any $D_M$, let positive $k, \lambda$ and $\omega_h$ satisfy \eqref{Ineq-Assumption-2} with $H=H_m$. Consider the estimation error system \eqref{Eq-dervative-tildetheta-Un-1}, \eqref{Eq-dervative-tildeeta-Un} with an unknown time-varying delay that satisfies \eqref{Eq-Delay-Cond}. The functions $M(j)$, $S(j)$ and $\alpha(j)$ are defined by \eqref{Eq-S(t)-M(t)} and \eqref{Eq-alpha} with tuning parameters $a_i$, $\alpha_0$ and $T$ satisfying \eqref{Eq-Omega}. Let $\bar{\rho}_{j}$, $j=1,2,3,4$ are the bounds defined in \eqref{Eq-rho-gamma-bounds}. Given any $\sigma_0 > 0$, let $\sigma $ be subject to $\sigma_0 < \sigma$. Let there exists $\varepsilon^*>0$ that satisfy
% \begin{equation} \label{Ineq-THM-Un-condition-1}
% n  D_M  \lfloor \frac{1}{\sqrt{ \varepsilon^*}} \rfloor \geq 3n,
% \end{equation}
%  \begin{equation} \label{Ineq-THM-Un-condition-2}
% \varepsilon^* \max \{ \lambda+kH_m, \omega_h + 2\lambda - \varepsilon^*\lambda^2  \}  < 2,
% \end{equation}
%\begin{equation} \label{Ineq-THM-Un-condition-3}
%\begin{array}{ll}
%& (\sigma_0 +\sqrt{\varepsilon^*} \Delta_G |1-\varepsilon^* \lambda|^{D_M}) |1-\varepsilon^* k H_m|^{-D_M}   \\
%& + \sqrt{\varepsilon^*} \Delta_G     + \frac{ (\sqrt{\varepsilon^*})^3 \Delta_Y  }{  |1-\varepsilon^* \lambda|-|1-\varepsilon^* k H_m| } < \sigma,
%\end{array}
%\end{equation}
\begin{subequations}
\begin{equation} \label{Ineq-THM-Un-condition-1}
 n  D_M  \lfloor \frac{1}{\sqrt{ \varepsilon^*}} \rfloor \geq 2n+1,
 \end{equation}
 \begin{equation} \label{Ineq-THM-Un-condition-2}
 \varepsilon^* \max \{ \lambda+kH_m, \omega_h + 2\lambda - \varepsilon^*\lambda^2  \}  < 2,
  \end{equation}
 \begin{equation}  \label{Ineq-THM-Un-condition-3}
 \!\!\!\!\!\!\!\!\!\!\!\!
 \begin{array}{ll}
& (\sigma_0 +\sqrt{\varepsilon^*} \Delta_G |1-\varepsilon^* \lambda|^{D_M}) |1-\varepsilon^* k H_m|^{-D_M}     \\
& \qquad \quad + \sqrt{\varepsilon^*} \Delta_G     + \frac{ (\sqrt{\varepsilon^*})^3 \Delta_Y  }{  |1-\varepsilon^* \lambda|-|1-\varepsilon^* k H_m| } < \sigma,
 \end{array}
\end{equation}
\end{subequations}
where
\begin{equation} \label{Eq-bounds-THM-1}
\!\!\!\!\!\! 
\begin{array}{ll}
& \Delta_{G}  =   \bar{\rho}_{1} \sigma + \bar{\rho}_{2}  \frac{\sigma_{\eta}}{\alpha_0}   + \bar{\rho}_3 \alpha_0   + \bar{\rho}_{4}  \frac{\sigma^2 H_M} {\alpha_0},    \\
& \sigma_{\eta}  =    \Delta_Q |1-\varepsilon^* \omega_h|^{-D_M}   +  \frac{  \varepsilon^*   \omega_h \sigma_y  } { |1-\varepsilon^* \lambda|^2 - |1-\varepsilon^* \omega_h| }, \\
&  \sigma_{y}  =      \frac{H_M}{2}  |1-\varepsilon^* \lambda|^{-2D_M} \left(  \sigma  + \alpha_0 \sqrt{\sum_{i=1}^n a_i^2}  \right) ^2,   \\
& \Delta_{Y}   =   k H_M \Delta_{G} + \bar{\rho}_1  (\Delta + { \sqrt{\varepsilon^*} \bar{\Delta}_{out} })  \\
&  + \frac{ \omega_h  \bar{\rho}_2 }{\alpha_0}  [\sigma_{y}+\sigma_{\eta}  ] +    \frac{ \bar{\rho}_{2}   \lambda ( \sigma_{\eta} |1-\varepsilon^* \omega_h|  +  \varepsilon^* \omega_h \sigma_y   )}{ \alpha_0|1-\varepsilon^* \lambda|}     \\
&  + \bar{\rho}_{3} \alpha_0  \lambda  + 2 \bar{\rho}_{4} \sigma \alpha_0^{-1}  H_M   (\Delta + { \sqrt{\varepsilon^*} \bar{\Delta}_{out}})    \\
&  \qquad  + {  \bar{\Delta}_{out} }  +  \varepsilon^* \bar{\rho}_{4}  \alpha_0^{-1}   H_M (\Delta + { \sqrt{\varepsilon^*} \bar{\Delta}_{out}})^2 \\
&  \qquad \qquad +\frac{ \bar{\rho}_{4}  \alpha_0^{-1}  \lambda}{|1-\varepsilon^* \lambda|} H_M  [\sigma + \varepsilon^* (\Delta + { \sqrt{\varepsilon^*} \bar{\Delta}_{out}})] ^2,\\
 &  \Delta  =    \frac{2 k}{\alpha_0} [\frac{H_M}{2}(\sigma+\alpha_0 \sqrt{\sum_{i=1}^n a_i^2} )^2 + \sigma_{\eta}]\sqrt{ \sum_{i=1}^n \frac{1}{  a_i^2} },  \\
 & {  \bar{\Delta}_{out}}  = \frac{k  D_M H_M}{  \alpha_0} \Big( \sqrt{ \sum_{i=1}^n \frac{1}{  a_i^2} } \Big) [\sigma+ \alpha_0 \sqrt{\sum_{i=1}^n a_i^2}]\times \\
 & \qquad   [1+ {|1-\varepsilon^*\lambda|^{-D_M}}]\Big[  \frac{2\sqrt{\varepsilon^*}k(\sigma_y+\sigma_{\eta})}{\alpha_0} \sqrt{ \sum_{i=1}^n \frac{1}{  a_i^2} } \\
 & \qquad   + \frac{\alpha_0\sqrt{\sum_{i=1}^n a_i^2} (\lambda\sqrt{\varepsilon^*}+\frac{2 \pi}{(1-\sqrt{\varepsilon^*})D_M}[\frac{ \pi \sqrt{\varepsilon^*}}{(1-\sqrt{\varepsilon^*})D_M}+1]) }{|1-\varepsilon^* \lambda|}  \Big].
 \end{array}
\end{equation}
Then, for all $\varepsilon \in (0,\varepsilon^*]$,  the following inequalities hold:
\begin{equation} \label{Ineq-THM-Un}
 |\tilde{\theta}(j)|   <   \sigma |1-\varepsilon \lambda|^j, \,\,\,  |\tilde{\eta}(j)|  < \sigma_{\eta} |1-\varepsilon \lambda|^{2j},   \,\,\,  j \geq D_M,
\end{equation}
meaning that the estimation error system \eqref{Eq-dervative-tildetheta-Un-1}, \eqref{Eq-dervative-tildeeta-Un} is exponentially stable with a decay rate $|1-\varepsilon \lambda|$. Thus, $\theta(j)$ converges exponentially to $\theta^*$ with a decay rate $|1-\varepsilon\lambda|$. Moreover, for any $D_M$ and $\sigma_0$, inequalities \eqref{Ineq-THM-Un-condition-1}, \eqref{Ineq-THM-Un-condition-2} and \eqref{Ineq-THM-Un-condition-3} are always feasible for small enough $\varepsilon^*$, $k$ and $\lambda$ and appropriate $\sigma > \sigma_0$.
\end{theorem}

{\bf Proof:}
See Appendix.
$\hfill \Box$

%{\color{red}
%\begin{remark}
%The following qualitative result is established by Theorem \ref{THM-1}:
%Following Theorem \ref{THM-1}, a qualitative result can be stated: Given any delay bound $D_M$, and parameters $a_i$, $k$, $\lambda$, $\omega_h$ satisfying \eqref{Ineq-Assumption-2}, there exist sufficiently small tuning parameters   $\varepsilon$ and $\omega_i=O(\sqrt{\varepsilon})$, $i=1,...,n$ such that the closed-loop system in (\ref{Eq-dervative-hattheta-Un-1}) and (\ref{Eq-dervative-eta-Un}) is exponentially stable at the origin. The latter means that the input
%$\theta(j)$ and output $y(j)$ exponentially converge to $\theta^*$ and $Q^*$, respectively.
%\end{remark}
%}

\begin{remark}\label{Rem-deacy-rate}
Differently from the continuous-time, the decay rates of $\tilde{\theta}_{av}$, $\tilde{\eta}$ and $\alpha(j)$ depend on $\varepsilon$, they are given by $|1-\varepsilon k H|$, $|1-\varepsilon \omega_h|$ and $|1-\varepsilon \lambda|$, respectively. Consequently, to guarantee the exponential convergence of \eqref{Eq-dervative-hattheta-Un-1}, \eqref{Eq-dervative-eta-Un}, the following conditions
\begin{equation} \label{Ineq-Assumption-2-result}
	  |1-\varepsilon \lambda|> |1- \varepsilon k H|,\,\,\,\, |1-\varepsilon \lambda|^2 > |1-\varepsilon \omega_h|,
\end{equation}
should be satisfied, which are equivalent to inequality \eqref{Ineq-THM-Un-condition-2}. %The meaning og these condition is that the first condition makes sure that the averaged system \eqref{Eq-Averaged-Sys} converges faster than the closed-loop \eqref{Eq-dervative-tildeeta-Un}, \eqref{Eq-dervative-tildetheta-Un-2} while the second condition makes sure that the filter  converges than the whole system.
These conditions ensure that both the averaged system \eqref{Eq-Averaged-Sys} and the filter dynamics \eqref{Eq-dervative-eta-Un} converge faster (i.e., possess smaller decay rates) than the whole closed-loop system \eqref{Eq-dervative-hattheta-Un-1}, \eqref{Eq-dervative-eta-Un}.
%The interpretation of these conditions is as follows: the first condition ensures that the averaged system \eqref{Eq-Averaged-Sys} converges faster than the closed-loop dynamics \eqref{Eq-dervative-tildeeta-Un}-\eqref{Eq-dervative-tildetheta-Un-2}, while the second condition ensures that the filter dynamics converge faster than the overall system.
Note that the conditions in \eqref{Ineq-Assumption-2-result} are reduced to \eqref{Ineq-Assumption-2} for $\varepsilon \to 0$.
\end{remark}

%\begin{remark}\label{Rem-Unbiased-Decay-Rate-2}
%Unlike the continuous-time case, the conditions \eqref{Ineq-THM-Un-condition-1}-\eqref{Ineq-THM-Un-condition-3} in Theorem \ref{THM-1} impose clear restrictions on the decay rate $\bar{\lambda}$. More precisely, for a given $\lambda$, larger delay $D_M$ leads to smaller $\varepsilon^*$, i.e. slower convergence and lower dither frequency. The same holds true for the classical ES (see next section) where the conditions \eqref{Ineq-Coro-Class-condition-1}-\eqref{Ineq-Coro-Class-condition-3} of Corollary \ref{THM-2} also impose restrictions on the decay rate $1-\varepsilon k H_m$.
%\end{remark}

\begin{remark}\label{Rem-Non-Quad}
If the map is a non-quadratic, yet $C^3$ differentiable function, it can be approximated by the quadratic form in \eqref{Eq-Quadratic-Form}. This approximation is valid in a neighborhood of $\theta^*$, specifically for
$|\theta(j)-\theta^*| \leq \sigma_1$ with some known $\sigma_1$. Consequently, the results from Theorem \ref{THM-1} hold true for $\sigma=\sigma_1-\alpha_0 \sqrt{\sum_{i=1}^n a_i^2}$.
Following the arguments of Theorem \ref{THM-1}, if conditions (\ref{Ineq-THM-Un-condition-1}), (\ref{Ineq-THM-Un-condition-2}), and (\ref{Ineq-THM-Un-condition-3}) are met, then the solutions of equations (\ref{Eq-dervative-tildetheta-Un-1}) and (\ref{Eq-dervative-tildeeta-Un}) are regionally exponentially stable for all $\varepsilon\le \varepsilon^*$, $D(j)$ subject to \eqref{Eq-Delay-Cond} and the initial condition $|\tilde{\theta}(0)| \leq \sigma_0$. %, i.e. $|\tilde{\theta}(j)| \leq \sigma e^{-\lambda t}$ with $\sigma<\sigma_1$.
Similar results hold true for the classical ES in the next section.
\end{remark}

\begin{remark}\label{Rem-Diff-UnB-Class-2}
Compared to the classical ES,  high-pass filter and exponential perturbation/demodulation signals in the unbiased ES lead to more challenging averaging-based stability analysis.
 Additional restrictions  \eqref{Ineq-Assumption-2-result} on the tuning parameters are needed  along with  the convergence  proof of the additional  error term ${\tilde{\eta}}(j)$.
To manage with the exponential convergence we assume and later prove that $ |\tilde\theta(j)| \leq \sigma \bar{\lambda}^j$ (see \eqref{Ineq-tildetheta-sigma-delay}) for some $\sigma>0$ instead of $|\tilde\theta(j)| \leq \sigma $ for the classical ES.
\end{remark}

\begin{remark}\label{Rem-Delay-Free-Un}
For delay-free case $D_M=0$, i.e. $D(j) \equiv 0$, less restrictive stability conditions than in Theorem \ref{THM-1} for the unbiased ES \eqref{Eq-dervative-tildetheta-Un-1}, \eqref{Eq-dervative-tildeeta-Un} can be derived, where from \eqref{Eq-Omega} it follows that $T=2n+1$ and the frequency $\omega_i={2\pi i \over 2n+1}$  is not slow. We have here $\rho_i=O(\epsilon), i=1,...,4$ (and not $\rho_i=O(\sqrt{\epsilon)}$):
\begin{equation} \label{Eq-rho-gamma-bounds-Delayfree}
\!\!\!\!\!\!\!\!\!
\begin{array}{ll}
&||\rho_{1}(j)|| \leq \varepsilon  \bar{\rho}_{1}  :=   \varepsilon \frac{k H_M}{2}   \Big[ \sum_{i=1 }^n   \frac{1 }{ |\sin(\omega_i)| }  \\
& \qquad \qquad    +   \sum_{i \neq l}^n  \frac{a_l}{a_i}  [   \frac{ 1 }{| \sin(\frac{\omega_i-\omega_l}{2})| }   +   \frac{ 1 }{ | \sin(\frac{\omega_i+\omega_l}{2})|  }  ]  \Big],\\
&|\rho_{2}(j)| \leq  \varepsilon  \bar{\rho}_{2}   :=   \varepsilon  k \left[ \sum_{i=1 }^n   \frac{1 }{ a_i |\sin(\frac{\omega_i}{2})| }   \right] ,\\
&  |\rho_{3}(j)| \leq \varepsilon \bar{\rho}_{3}  := \varepsilon k  H_M  \left( \sum_{s=1}^{n}a_s^2\right) \sqrt{\sum_{i=1}^n \frac{\bar{c}_i^2}{a_i^2} }  ,  \\
&|\rho_{4}(j)| \leq  \varepsilon  \bar{\rho}_{4}   :=   \varepsilon  \frac{k}{2} \left[ \sum_{i=1 }^n   \frac{1 }{ a_i |\sin(\frac{\omega_i}{2})| }   \right],
\end{array}
\end{equation}
where
\begin{equation*}
\begin{array}{ll}
& \bar{c}_i =   \Big[  \frac{ ( n + 1)(4n+3) }{ 6 }   \\
& \qquad \qquad + \frac{ 1  }{ 4 |\sin(\omega_i)| }  \left( |\cot(\omega_i)| + 2n + 2 + \frac{ 1 }{ 2n+1 }  \right)   \Big].
\end{array}
\end{equation*}
Let there exists $\varepsilon^*$ that satisfy inequality \eqref{Ineq-THM-Un-condition-2} and  \begin{equation}
\sigma_0 +2 \varepsilon^*  \Delta_G   + \frac{ ( \varepsilon^* )^2 \Delta_Y  }{  |1-\varepsilon^* \lambda|-|1-\varepsilon^* k H_m| } < \sigma,    \label{Ineq-THM-Un-condition-DelayFree}
\end{equation}
where the upper-bounds $\Delta_G$, $\Delta_Y$, $\Delta$, $\sigma_y$, $\sigma_{\eta}$ are defined in \eqref{Eq-bounds-THM-1} with $D_M=\bar{\Delta}_{out}=0$. Then, by arguments  of Theorem \ref{THM-1},  inequalities \eqref{Ineq-THM-Un} hold for all $\varepsilon \in (0,\varepsilon^*]$.
\end{remark}

For completely unknown maps, the following  qualitative stability result follows from Theorem \ref{THM-1}:
\begin{corollary}\label{Coro-1}
Given any $\sigma_0$, $D_M$ and positive tuning parameters  $a_i$, $k$, $\lambda$, $\omega_h$ satisfying \eqref{Ineq-Assumption-2}, there exist sufficiently small tuning parameters   $\varepsilon$ and $\omega_i=O(\sqrt{\varepsilon})$, $i=1,...,n$ such that the estimation error  system \eqref{Eq-dervative-tildetheta-Un-1}, \eqref{Eq-dervative-tildeeta-Un} is exponentially stable.
%The latter means that the input $\theta(j)$ and output $y(j)$ exponentially converge to $\theta^*$ and $Q^*$, respectively.
\end{corollary}

%%%%%%%%%%%%%%%%%%%%%%%%%%%%%%%%%%%%%%%%%%%%%%%%%%%%%%%%%%%%%%%%%%%%%%%%%%%%%%%%%%%%%%%%%%%%%%%%%%%%%%%%%%%%%%%%%%%%%%%%%%%%%%%%%%%%%%%%%%%%%%%%%%%%%%%%%%%%%%%%%%%%%%%%%%%%%%%%%%%%%%%%%%%%%%%%

 \section{Classical ES with unknown delays} \label{Sec-Bounded-ES}

In this section, we consider the classical ES algorithm  in the presence of unknown large time-varying  measurement delay. Consider the ES algorithm \eqref{Eq-y-Q}, \eqref{Eq-hattheta-Un}, \eqref{Eq-dervative-hattheta-Un-1} with $\alpha(j) \equiv 1$  and $\eta(j) \equiv 0$ for all $j \geq 0$, which results in the following
equation for the real-time estimate $\hat{\theta}(j)$:
\begin{equation}\label{Eq-hattheta-derivative-Class}
\begin{array}{ll}
&  \hat{\theta}(j+1) =     \hat{\theta}(j) - \varepsilon   kM(j)y(j), \quad j \geq D_M, \\
& \hat{\theta}(j)=\theta(0), \quad  0 \leq j \leq D_M,
 \end{array}
\end{equation}
where $S(j)$, $M(j)$ are defined in \eqref{Eq-S(t)-M(t)}, $k$ is a positive gain. % and $D(j)$ is an unknown bounded time-varying delay satisfying \eqref{Assumption-D}.
Then, the estimation error $\tilde{\theta}(j)=\hat{\theta}(j)-\theta^*$, $j \geq 0$ is governed by
\begin{equation} \label{Eq-derivative-tildetheta-Class}
\!\!\!\!\!\!
\begin{array}{ll}
&    \tilde{\theta}(j+1) = \tilde{\theta}(j)  - \varepsilon    k M(j)\Big[ Q^* \\
  & \qquad  \qquad \quad + \frac{1}{2} |\tilde{\theta}(j-D(j))  +  S(j)|^2_{ H }  \Big],  \,\,\,\,  j \geq  D_M, \\
  & |\tilde{\theta}(j)| \leq \sigma_0,  \qquad \qquad \qquad \qquad \qquad 0 \leq j \leq D_M.
 \end{array}
\end{equation}
%Consider the closed-loop system \eqref{Eq-hattheta-derivative}, \eqref{Eq-eta-derivative} with $\alpha_0=1$, $\lambda=0$, $\omega_h=0$ and $\eta(0)=-Q^*$, then \eqref{Eq-tildetheta-derivative}, \eqref{Eq-tildeeta-derivative} is reduced to
%\begin{align}\label{Eq-tildethtea-derivative-bound}
% \dot{\tilde{\theta}}(t) =  - K M(t)  \left(  Q^* + \frac{H}{2}  [ \tilde{\theta}(t) +  S(t)]^2   \right),    |\tilde{\theta}(0)| \leq \sigma_0
%\end{align}
%Note that \eqref{Eq-tildetheta-derivative-Class} with $D(j) = 0$ is the estimation error in the discrete classical ES algorithm which was analyzed in \cite{yang2023robust} via the time-delay.
By similar arguments of Theorem \ref{THM-1}, we derive the following result:
\begin{theorem}\label{THM-2}
Let Assumptions \ref{Assumption-Hessian} and \ref{Assumption-theta} hold. Given
any $D_M>0$, consider the estimation error system \eqref{Eq-derivative-tildetheta-Class} with an unknown time-varying delay $D(j)$ that satisfies \eqref{Eq-Delay-Cond}. The functions $M(j)$ and $S(j)$ are defined by \eqref{Eq-S(t)-M(t)} with tuning parameters $a_i$ and $T$ satisfying \eqref{Eq-Omega}. Let $\bar{\rho}_{j}$, $j=1,2,3,4$ be the bounds defined in \eqref{Eq-rho-gamma-bounds}. Given any $\sigma_0 > 0$, let $\sigma $ be subject to $\sigma_0 < \sigma$. Let there exists $\varepsilon^*$ that satisfy
\begin{subequations}
\begin{equation}\label{Ineq-Coro-Class-condition-1}
 n  D_M \lfloor \frac{1}{\sqrt{ \varepsilon^*}} \rfloor \geq 2n+1,
 \end{equation}
 \begin{equation}\label{Ineq-Coro-Class-condition-2}
  \varepsilon^*  kH_m  < 2,
  \end{equation}
 \begin{equation}\label{Ineq-Coro-Class-condition-3}
 \!\!\!\!\!\!
\begin{array}{ll}
 (\sigma_0 +\sqrt{\varepsilon^*} \Delta_G  ) |1-\varepsilon^* k H_m|^{-D_M}  \\
 \qquad \qquad \quad  + \sqrt{\varepsilon^*} \Delta_G     + \frac{ (\sqrt{\varepsilon^*})^3 \Delta_Y  }{1 -|1-\varepsilon^* k H_m| } < \sigma,
\end{array}
\end{equation}
\end{subequations}
where
\begin{equation}  \label{Eq-bounds-Coro-1}
\!\!\!\!\!\! 
\begin{array}{ll}
& \Delta_{G}  =   \bar{\rho}_{1} \sigma + \bar{\rho}_{2} (Q_0 + \Delta_Q)  + \bar{\rho}_3   + \bar{\rho}_{4}  \sigma^2 H_M,    \\
&  \sigma_{y}  =      \frac{H_M}{2}   \left(  \sigma  +  \sqrt{\sum_{i=1}^n a_i^2}  \right) ^2,   \\
& \Delta_{Y}   =   k H_M \Delta_{G} + \bar{\rho}_1  (\Delta + { \sqrt{\varepsilon^*} \bar{\Delta}_{out} })  \\
&  \qquad\quad + 2 \bar{\rho}_{4} \sigma   H_M   (\Delta + { \sqrt{\varepsilon^*} \bar{\Delta}_{out}})   + {   \bar{\Delta}_{out} }   \\
& \qquad \qquad  +  \varepsilon^* \bar{\rho}_{4}    H_M (\Delta + { \sqrt{\varepsilon^*} \bar{\Delta}_{out}})^2, \\ 
 &  \Delta  =     2 k  [ \sigma_{y}  +  Q_0+\Delta_Q ]\sqrt{ \sum_{i=1}^n \frac{1}{  a_i^2} },  \\
 & {  \bar{\Delta}_{out}}  =    2 k  D_M H_M  \Big( \sqrt{ \sum_{i=1}^n \frac{1}{  a_i^2} } \Big) [\sigma+   \sqrt{\sum_{i=1}^n a_i^2}]\times \\
 & \qquad  \qquad  \Big[ \sqrt{\varepsilon^*}   \Delta  +    \frac{2 \pi  \sqrt{\sum_{i=1}^n a_i^2} }{(1-\sqrt{\varepsilon^*})D_M}[\frac{ \pi \sqrt{\varepsilon^*}}{(1-\sqrt{\varepsilon^*})D_M}+1]    \Big].
 \end{array}
\end{equation}
Then, for all $\varepsilon \in (0,\varepsilon^*]$ the solutions of \eqref{Eq-derivative-tildetheta-Class} satisfy $|\tilde{\theta}(j)|    < \sigma, \,\, j   \geq  D_M$ and are exponentially attracted to the set
\begin{equation} \label{Ineq-Ultimate-Bound}
\!\!\!\!
\begin{array}{ll}
\Theta = \left\lbrace  \tilde{\theta}(j) \in \mathbb{R} : |\tilde{\theta}(j)|  < \sqrt{\varepsilon }   \left(  \Delta_G     + \frac{   \varepsilon   \Delta_Y  }{1- |1-\varepsilon  k H_m| } \right)   \right\rbrace 
 \end{array}
\end{equation}
with a decay rate {  $|1-\varepsilon k H_m|$.} Moreover, for any $D_M$ and $\sigma_0$, inequalities \eqref{Ineq-Coro-Class-condition-1}, \eqref{Ineq-Coro-Class-condition-2} and \eqref{Ineq-Coro-Class-condition-3} are always feasible for small enough $\varepsilon^*$, $k$ and appropriate $\sigma > \sigma_0$.
\end{theorem}

%{\color{red}
%\begin{remark}
%The following qualitative result is established by Theorem \ref{THM-1}:
%Following Corollary \ref{THM-2}, a qualitative result can be stated: Given any delay bound $D_M$, and parameters $a_i$, $k$, there sufficiently small tuning parameters $\varepsilon$ and $\omega_i=O(\sqrt{\varepsilon})$, $i=1,...,n$ such the closed-loop system \eqref{Eq-hattheta-derivative-Class} is practically stable at the origin.
%\end{remark}
%}

\begin{remark}\label{Rem-Delay-Free-Class}
Similar to Remark \ref{Rem-Delay-Free-Un}, for the delay-free case $D_M=0$,  the classical ES \eqref{Eq-derivative-tildetheta-Class} is applicable with $\omega_i={2\pi i \over 2n+1}$.  Let there exists $\varepsilon^*$ that satisfies inequalities \eqref{Ineq-Coro-Class-condition-2} and \eqref{Ineq-THM-Un-condition-DelayFree} with $\lambda=0$, where $\bar{\rho}_i$, $i=1,2,3,4,$ $\Delta_G$, $\Delta_Y$, $\Delta$, $\sigma_y$ are defined by  \eqref{Eq-rho-gamma-bounds-Delayfree}, \eqref{Eq-bounds-Coro-1} with $D_M=\bar{\Delta}_{out}=0$. Then, for all $\varepsilon \in (0,\varepsilon^*]$, the solutions of \eqref{Eq-derivative-tildetheta-Class} are exponentially attracted to the set
\begin{equation} \label{Ineq-Ultimate-Bound-Delayfree}
\!\!\!\!
\begin{array}{ll}
\Theta = \left\lbrace  \tilde{\theta}(j) \in \mathbb{R} : |\tilde{\theta}(j)|  < \varepsilon   \left(  \Delta_G     + \frac{   \varepsilon   \Delta_Y  }{1- |1-\varepsilon  k H_m| } \right)   \right\rbrace 
\end{array}
\end{equation}
with a decay rate $|1-\varepsilon k H_m|$.
\end{remark}

For completely unknown maps, the following  qualitative stability result follows from Theorem \ref{THM-2}:
\begin{corollary}\label{Coro-2}
Given any $\sigma_0$, $D_M$, and positive tuning parameters $a_i$, $k$, there exist sufficiently small tuning parameters $\varepsilon$ and $\omega_i=O(\sqrt{\varepsilon})$, $i=1,...,n$ such that the estimation error system  \eqref{Eq-derivative-tildetheta-Class} is practically stable.% at the origin.
\end{corollary}

\section{Example}\label{Sec-Example}

 Consider the $3D$ map \eqref{Eq-Quadratic-Form} with
\begin{equation}\label{Eq-MAP}
\theta^*= [ 2 \quad  4 \quad  1]^T, \quad Q_0=0, \quad  H=   \begin{bmatrix}
  100 & 30 & 5 \\
  30 & 20 & 5  \\
   5 & 5 & 50
  \end{bmatrix}.
\end{equation}
Note that the eigenvalues of $H$ are $9.749$ and $110.66$. Also notice that this map with $\theta^*_3=1$ coincides with the $2$D map in \cite{guay2014time}. We assume that $Q^*$ and $H$ are uncertain satisfying Assumptions \ref{Assumption-Hessian}, \ref{Assumption-theta} with a larger interval for $H$
\begin{equation}\label{Eq-Case-Parameters}
\Delta_Q = 1, \qquad H_m=9.5, \qquad H_M=111.
\end{equation}
%We select the tuning parameters of the ES algorithm as $a = 0.5$ and $k=0.05$. %Let $D(j)$ be an unknown time-varying bounded delay satisfying \eqref{Eq-Delay-Cond} with known bound $D_M$ (will be selected later).time-varying delay $D(j) = D_M \cdot \text{randi}([0, 1])$, where $\text{randi}([a, b])$ is a MATLAB function that generates a random integer from a uniform distribution between $a$ and $b$.
%Let $D(j)$ be an unknown time-varying delay satisfying \eqref{Eq-Delay-Cond} with a known bound $D_M$ (to be selected later), chosen as $D(j) = \text{randi}([0, D_M ])$, where $\text{randi}([a, b])$ is a MATLAB function that generates a random integer from a uniform distribution between $a$ and $b$.
  We further present  results that follow from Theorems \ref{THM-1} and \ref{THM-2}, where the parameter $\sigma$ is tuned to achieve as large as possible $\varepsilon^*$.
  We also find maximum $\epsilon^*$ that follows from simulations, where we choose $D(j) = \text{randi}([0, D_M])$ with % subject to $D(j)\le D_M$. Here,
MATLAB function $\text{randi}([0, D_M])$ generating a random integer from a discrete uniform distribution on $[0, D_M]$.

\underline{\textit{ES with exponential stability:}}
Consider %the estimation errors \eqref{Eq-dervative-tildetheta-Un-1}, \eqref{Eq-dervative-tildeeta-Un} with
\begin{subequations}
\begin{equation}\label{Eq-Example-Para-1}
\sigma_0=1, \,\, a = 0.1, \,\,  k=0.005,
 \end{equation}
 \begin{equation}\label{Eq-Example-Para-2}
\lambda=0.005,\,\, \omega_h=0.015.
\end{equation}
\end{subequations}
%for map \eqref{Eq-Quadratic-Form}, \eqref{Eq-MAP}. 
The maximum values of $\varepsilon^*$ that guarantee the stability of  the estimation error system \eqref{Eq-dervative-tildetheta-Un-1}, \eqref{Eq-dervative-tildeeta-Un} are presented in Table \ref{Table-3D-Un}, comparing the analytical bounds derived from Theorem \ref{THM-1} with the simulation-based results.
\begin{table}[h]
 \centering
 \scalebox{0.59} {  \begin{tabular}{ | c | c | c | c | c | c |     }
 		\hline
 		Cases   &  $D_M$   &  $\sigma$ &    $\varepsilon^*$ &  $\bar{\lambda} = 1- \varepsilon^*\lambda$   \\ [1mm]
 		\hline
 		 \multirow{3}{*}{Theorem \ref{THM-1} with $H_m$, $H_M$, $\Delta_Q$ from \eqref{Eq-Case-Parameters}}
 & \multicolumn{1}{c|}{0}    & \multicolumn{1}{c|}{1.6}  & \multicolumn{1}{c|}{$0.43 \cdot 10^{-5}$}  & \multicolumn{1}{c|}{ $1-0.21 \cdot 10^{-7}$ }  \\
 \cline{2-5}
 & \multicolumn{1}{c|}{5}   & \multicolumn{1}{c|}{1.7}  & \multicolumn{1}{c|}{$0.36  \cdot 10^{-11}$  }  & \multicolumn{1}{c|}{ $1-0.16 \cdot 10^{-13}$ }  \\
 \cline{2-5}
 & \multicolumn{1}{c|}{50}   & \multicolumn{1}{c|}{1.7}  & \multicolumn{1}{c|}{$0.3 \cdot 10^{-13}$  }  & \multicolumn{1}{c|}{ $1-0.15 \cdot 10^{-15}$ }  \\
\hline
 		 \multirow{3}{*}{Simulation}
 & \multicolumn{1}{c|}{0}    & \multicolumn{1}{c|}{-}  & \multicolumn{1}{c|}{$0.7 \cdot 10^{-3}$}  & \multicolumn{1}{c|}{ $1-0.35 \cdot 10^{-5}$ }  \\
 \cline{2-5}
 & \multicolumn{1}{c|}{5}   & \multicolumn{1}{c|}{-}  & \multicolumn{1}{c|}{$0.3 \cdot 10^{-4}$  }  & \multicolumn{1}{c|}{ $1-0.15 \cdot 10^{-6}$ }  \\
 \cline{2-5}
 & \multicolumn{1}{c|}{50}   & \multicolumn{1}{c|}{-}  & \multicolumn{1}{c|}{$0.3 \cdot 10^{-5}$  }  & \multicolumn{1}{c|}{ $1-0.15 \cdot 10^{-7}$ }  \\
\hline
 		
\end{tabular} }

\caption{Maximum $\varepsilon^*$ for unbiased ES  with \eqref{Eq-Example-Para-1}, \eqref{Eq-Example-Para-2}.}   \label{Table-3D-Un}

% \caption{Unbiased ES: Maximum $\varepsilon^*$ for $\sigma_0=1$.} \label{Table-Delay}
 %\setlength{\extrarowheight}{2pt}	

  \end{table}

% The maximum values of $\varepsilon^*$ that follows once from Simulations and once from Theorem \ref{THM-1} for uncertain $Q^*$, $H$, and known $D_M$ are shown in Table \ref{Table-3D-Un}:

%Note that in the next obtained results from Theorems \ref{THM-1} and \ref{THM-2}, the parameter $\sigma$ is tuned to achieve as large as possible $\varepsilon^*$.

  %\vspace*{-\baselineskip} %omit space

%We further provide simulation of the  unbiased ES algorithm \eqref{Eq-dervative-tildetheta-Un-1}, \eqref{Eq-dervative-tildeeta-Un} for map \eqref{Eq-Quadratic-Form}, \eqref{Eq-MAP} with an initial condition $\theta(0) = [1.3,3.5,0.5]^T$, a time-varying delay $D(j)=  randi([0\,\,\,  5])$ ($randi[a \,\, b]$ is an Matlab function which generates a random integer from the uniform distribution between $a$ and $b$), $k=\lambda=0.05$ and $\omega_h= 0.15$. According to the simulations, $\varepsilon^*$ is $20$ with a decay rate $\bar{\lambda} =1-1\cdot 10^{-4} = 0.9999$.
%The plot of $|\hat{\theta}(t)-\theta^*|$ with initial condition $\theta(0) = [1.3,3.5,0.5]^T$( i.e. $\sigma_0=1$) and $D_M=0,5$, $\varepsilon^*=0.5 \cdot 10^{-2}, 0.8 \cdot 10^{-5}$ recpectivly is presented in Fig. \ref{Fig-DUnES}, where it can be seen that $\hat{\theta}(t)$ converges exponentially to the extremum point $\theta^*$.
The trajectories of $|\hat{\theta}(t)-\theta^*|$ for the initial condition $\hat{\theta}(0) = [1.3, 3.5, 0.5]^T$ (i.e., $\sigma_0=1$) and tuning parameters \eqref{Eq-Example-Para-1}, \eqref{Eq-Example-Para-2} with $(D_M, \varepsilon^*) = (0, 0.7 \cdot 10^{-3})$ and $(5, 0.3 \cdot 10^{-4})$ are shown in Fig. \ref{Fig-DUnES}, demonstrating exponential convergence of $\hat{\theta}(t)$ to $\theta^*$. While this confirms the efficiency of the unbiased ES algorithm, it also mirrors the conservatism of the analytical bounds on $\varepsilon^*$.
%is presented in Fig. \ref{Fig-DUnES} where it can be seen that $\hat{\theta}(t)$ exponentially converges to the extremum point $\theta^*$. This confirms the efficiency of the unbiased ES algorithm but also suggests that our theoretical stability bound on $\varepsilon^*$ as appeared in Table \ref{Table-3D-Un} is quite conservative compared to actual performance.

%. It is seen that ${\theta}(t)$ exponentially converges to the extremum point ${\theta}^*$, which demonstrates the efficiency of the method. In addition, this shows the consrviatsim of our quantitive results.

%Fig. \ref{Fig-DUnES} illustrates the evolution of $|\theta(t)-\theta^*|$, demonstrating the exponential convergence of $\theta(t)$ to the optimal point $\theta^*$ and confirming the effectiveness of the control scheme.

\begin{figure}[h!]
\includegraphics[width=8.5cm, height=6cm]{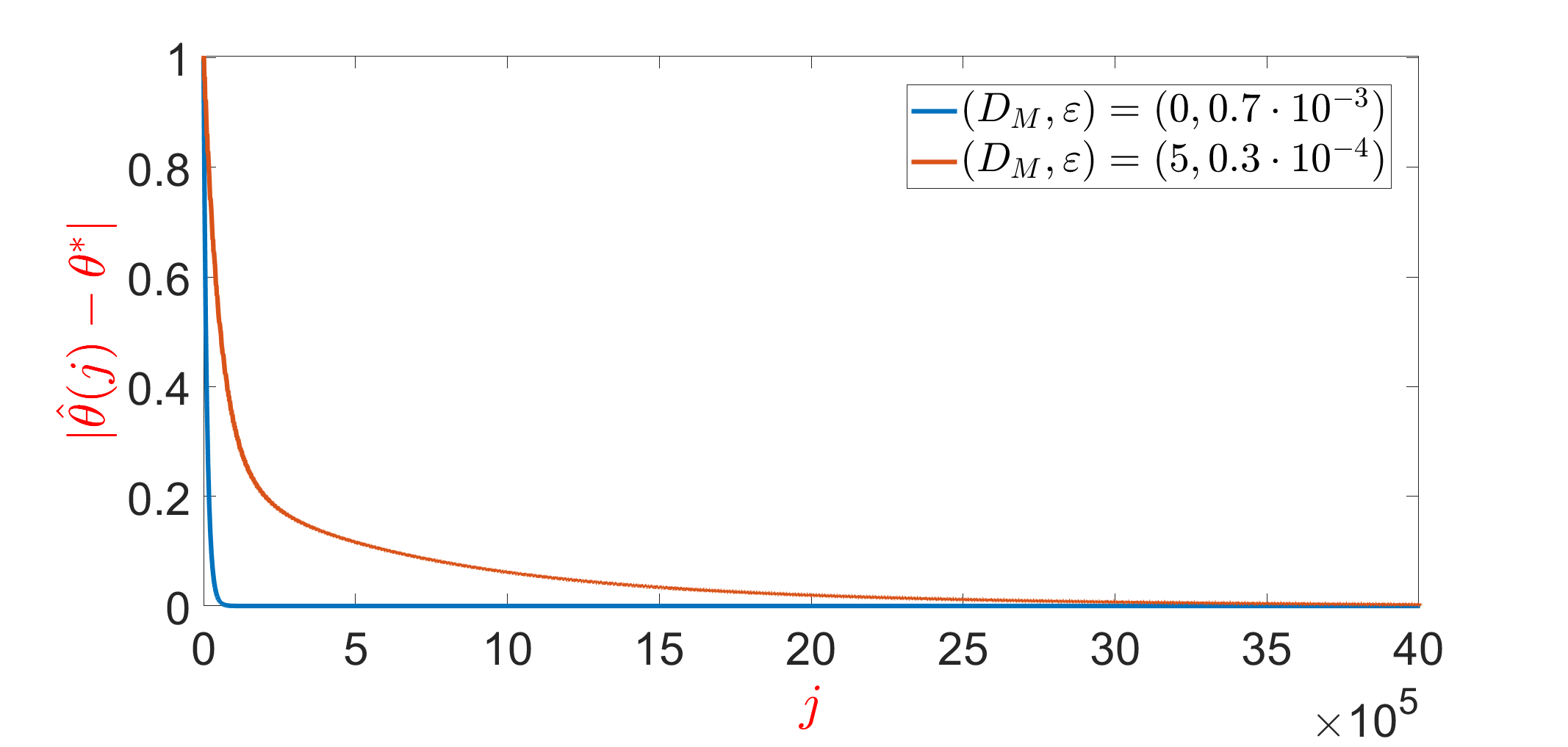}
\centering
\caption{Unbiased ES \eqref{Eq-dervative-tildetheta-Un-1}, \eqref{Eq-dervative-tildeeta-Un} with  \eqref{Eq-Example-Para-1}, \eqref{Eq-Example-Para-2} for $(D_M, \varepsilon^*) = (0, 0.7 \cdot 10^{-3})$ and $(5, 0.3 \cdot 10^{-4})$.
}\label{Fig-DUnES}
\end{figure}

\underline{\textit{ES with practical stability:}} Consider $\sigma_0$, $a$ and $k$ given by \eqref{Eq-Example-Para-1}. The maximum values of $\varepsilon^*$ for the estimation errors \eqref{Eq-derivative-tildetheta-Class} are presented in Table \ref{Table-3D-Cl}, comparing analytical bounds that follow from Theorem \ref{THM-2} with simulation-based results.

%\begin{table}[H]
 %	\centering
 %  \small
% \setlength{\extrarowheight}{2pt}
% \scalebox{0.7} { 	\begin{tabular}{ | c | c | c | c | c | c |      }
 %		\hline
 %		Cases &       $\sigma$ &  $\delta$ &   $\varepsilon^*$   &    UB \\ [1mm]
 %		\hline
 %		   Corollary 1 in \cite{yang2023time},   $\Delta_Q=0$, $H_m=H_M = 2$      &    $\sqrt{2}$   &      $0.013$   &    $0.079$  &  $ - $       \\[1mm]
 %		\hline
 %		Corollary \ref{Corollary-1} (our result),   $\Delta_Q=0$, $H_m=H_M = 2$  & $1.6$ &    $0.013$  &    $0.4805$ &  $0.5737$         \\[1mm]
 %		\hline
 %		Corollary 1 in \cite{yang2023time},   $\Delta_Q=0.1$, $H_m=1.9$,  $H_M = 2.1$   &    $\sqrt{2}$   &    $0.012$   &    $0.072$  &  $ - $       \\[1mm]
 %		\hline
% 		Corollary \ref{Corollary-1} (our result),   $\Delta_Q=0.1$, $H_m=1.9$,  $H_M = 2.1$   & $1.6$  &    $0.013$   &    $0.4053$ &  $0.5759$        \\[1mm]
 %		\hline
 %		 Corollary 1 in \cite{yang2023time},   $\Delta_Q=1$, $H_m=1.6$,  $H_M = 7.9$   &    $\sqrt{2}$  &    $0.010$   &    $0.018$  &  $ 5.3 \cdot 10^{-3}$         \\[1mm]
% 		\hline
% 		Corollary 1 (our result),   $\Delta_Q=1$, $H_m=1.6$,  $H_M = 7.9$   & $1.6$ &    $0.013$   &  $0.0245$   &    $0.5942$     \\[1mm] % 0.0245 $0.0263$
% 		\hline
% 	\end{tabular}
% 	}
 %	\medskip
% \caption{Comparison of maximum $\varepsilon^*$ for $\sigma_0=1$ and $D=0$.}\label{Table-Bound}

% 	\normalsize
 	
% \end{table}

%

 \begin{table}[h]
 \centering
 \scalebox{0.59} {  \begin{tabular}{ | c | c | c | c | c | c |     }
 		\hline
 		Cases   &  $D_M$   &  $\sigma$ &    $\varepsilon^*$ &  $\bar{\lambda} = 1- \varepsilon^*kH_m$   \\ [1mm]
 		\hline
 		 \multirow{3}{*}{Theorem \ref{THM-2} with $H_m$, $H_M$, $\Delta_Q$ from \eqref{Eq-Case-Parameters}}
 & \multicolumn{1}{c|}{0}    & \multicolumn{1}{c|}{1.6}  & \multicolumn{1}{c|}{$0.19 \cdot 10^{-4}$}  & \multicolumn{1}{c|}{ $1-0.9 \cdot 10^{-6}$ }  \\
 \cline{2-5}
 & \multicolumn{1}{c|}{5}   & \multicolumn{1}{c|}{1.6}  & \multicolumn{1}{c|}{$0.66  \cdot 10^{-10}$  }  & \multicolumn{1}{c|}{ $1-0.31 \cdot 10^{-11}$ }  \\
 \cline{2-5}
 & \multicolumn{1}{c|}{50}   & \multicolumn{1}{c|}{1.6}  & \multicolumn{1}{c|}{$0.63 \cdot 10^{-12}$  }  & \multicolumn{1}{c|}{ $1-0.3 \cdot 10^{-13}$ }  \\
\hline
 		 \multirow{3}{*}{Simulation}
 & \multicolumn{1}{c|}{0}    & \multicolumn{1}{c|}{-}  & \multicolumn{1}{c|}{$0.2 \cdot 10^{-2}$}  & \multicolumn{1}{c|}{ $1-0.95 \cdot 10^{-4}$ }  \\
 \cline{2-5}
 & \multicolumn{1}{c|}{5}   & \multicolumn{1}{c|}{-}  & \multicolumn{1}{c|}{$0.2 \cdot 10^{-3}$  }  & \multicolumn{1}{c|}{ $1-0.95 \cdot 10^{-5}$ }  \\
 \cline{2-5}
 & \multicolumn{1}{c|}{50}   & \multicolumn{1}{c|}{-}  & \multicolumn{1}{c|}{$0.8 \cdot 10^{-5}$  }  & \multicolumn{1}{c|}{ $1-0.38 \cdot 10^{-6}$ }  \\
\hline
 		
 	\end{tabular}
 	}
 	\medskip
 \caption{Maximum $\varepsilon^*$ for classical ES \eqref{Eq-derivative-tildetheta-Class} with \eqref{Eq-Example-Para-1}}\label{Table-3D-Cl}

 	\normalsize
 	
 \end{table}

%In addition, we provide a simulation of the classical ES algorithm \eqref{Eq-derivative-tildetheta-Class} for map \eqref{Eq-Quadratic-Form}, \eqref{Eq-MAP}, with the same initial condition, time-varying delay, and tuning parameters as previously. While the simulations yield $\varepsilon^* = 150$ with a decay rate $\bar{\lambda} = 1-\varepsilon kH_m \approx 0.992$, it also highlights the conservatism of our theoretical stability bound on $\varepsilon^*$. As shown in Fig. \ref{Fig-DClassES}, $\hat{\theta}(t)$ converges to a small neighborhood of the extremum point $\theta^*$. The convergence requires approximately $400$ iterations which is consistent with established results in the literature, such as \cite{discrete-LQ-control-linear-extremum-seeking-Krstic}.

The trajectories of $|\hat{\theta}(t)-\theta^*|$ for the initial condition $\hat{\theta}(0) = [1.3, 3.5, 0.5]^T$ (i.e., $\sigma_0=1$) and tuning parameters \eqref{Eq-Example-Para-1}, with $(D_M, \varepsilon^*) = (0, 0.2 \cdot 10^{-2})$ and $(5, 0.2 \cdot 10^{-3})$, are shown in Fig. \ref{Fig-DClassES}. As illustrated, $\hat{\theta}(t)$ converges to a small neighborhood of $\theta^*$. Convergence requires approximately $50 \cdot 10^4$ iterations, consistent with literature such as \cite{discrete-LQ-control-linear-extremum-seeking-Krstic, ES-Distributed-Optimization2025}. Furthermore, Table \ref{Table-3D-Cl} highlights the conservatism of our theoretical  bound $\varepsilon^*$ for classical ES.

\begin{figure}[h!]
\includegraphics[width=8.5cm, height=6cm]{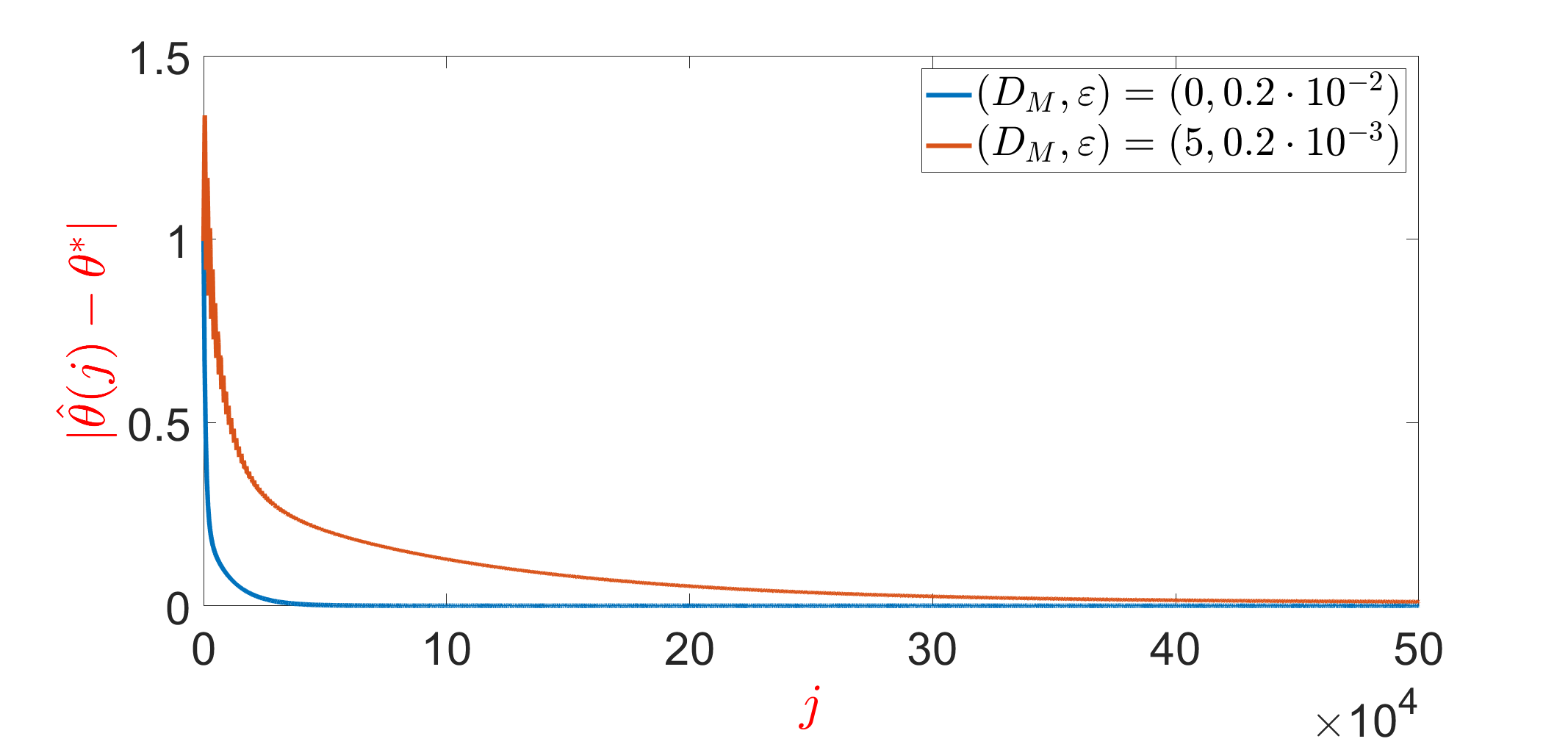}
\centering
\caption{Classical ES \eqref{Eq-derivative-tildetheta-Class} with \eqref{Eq-Example-Para-1} for $(D_M, \varepsilon^*) = (0, 0.2 \cdot 10^{-2})$ and $(5, 0.2 \cdot 10^{-3})$. %A zoom plot is also provided.
}\label{Fig-DClassES}
\end{figure}

From Tables \ref{Table-3D-Un}-\ref{Table-3D-Cl}, it is seen that larger delay bound $D_M$ lead to a smaller $\varepsilon^*$ and smaller decay rate $\bar{\lambda}$, i.e. higher dither frequencies and a slower convergence (as demonstrated in Figs. \ref{Fig-DUnES} and \ref{Fig-DClassES}). Note that the values of $ \varepsilon^*$  obtained by the unbiased ES algorithm (both the ones that follow from Theorems and  from simulations) are smaller than in the classical ES. 

In practice, the difference between the theoretical and simulation results suggests a two-step tuning procedure. First, Theorems \ref{THM-1} or \ref{THM-2} provide tuning parameters that maximize $\varepsilon^*$, bypassing exhaustive simulations. Subsequently, with these tuning parameters fixed, $\varepsilon^*$ can be further %incrementally 
enlarged
 via simulations to optimize performance before online implementation. This strategy utilizes theoretical results for safe initialization while achieving faster convergence and better performance.

\section{Conclusion}
This paper presented the first ES algorithm, which is robust with respect to arbitrary large unknown time-varying bounded delay. We presented discrete-time unbiased ES algorithm for uncertain $n$D quadratic maps in the presence of unknown time-varying bounded delays. % via the delay-free transformation.
%We established explicit, constructive conditions in the form of scalar linear inequalities that guarantee the exponential stability of the ES control system. Furthermore, we establish constructive conditions for the practical stability of the classical ES system. Our results are semi-global for globally quadratic maps, whereas for locally quadratic static maps, we provide a bound on the region of convergence. The constructive nature of our conditions is significant, as it allows for the selection of appropriate ES parameters to achieve convergence for any large unknown time-varying delay bounded by a known constant.
 Robustness is achieved due to the choice of 
 dithers with  frequencies of the order of $\sqrt{\varepsilon}$, where
the small parameter $\varepsilon>0$ appears in the dynamics of the real-time estimator. Larger delays lead to a slower convergence. 
 
 A potential direction for future research is the development of constructive methods for unbiased ES of dynamic systems in the presence of time-varying delays.

%%%%%%%%%%%%%%%%%%%%%%%%%%%%%%%%%%%%%%%%%%%%%%%%%%%%%%%%%%%%%%%%%%%%%%%%%%%%%%%%%%%%%%%%%%%%%%%%%%%%%%%%%%%%%%%%%%%%%%%%%%%%%%%%%%%%%%%%%%%%%%%%%%%%%%%%%%%%%%%%%%%%%%%%%%%%%%%%%%%%%%%%%%%%%%%%

\section{Appendix}

%%%%%%%%%%%%%%%%%%%%%%%%%%%%%%%%%%%%%%%%%%%%%%%%%%%%%%%%%%%%%%%%%%%%%%%%%%%%%%%%%%%%%%%%%%%%%%%%%%%%%%%%%%%%%%%%%%%%%%%%%%%%%%%%%%%%%%%%%%%%%%%%%%%%%%%%%%%%%%%%%%%%%%%%%%%%%%%%%%%%%%%%%%%%%%%%

{\bf Proof of Lemma \ref{Lem-Upper-Bounds}:}
The function $\rho_{1}(j)$ can be rewritten as
\begin{equation}\label{Eq-rho1-1}
 \!\!\!\! \!\!\!\!
\begin{array}{ll}
& \rho_{1}(j)  =  k \Big( \frac{\varepsilon}{T} \sum_{s=j}^{j+ T-1} (j+T-s)   [M(s)S^T(s) -  I]  \Big) H,
\end{array}
\end{equation}
with
%\begin{equation}\label{Eq-rho1-2}
%\!\!\!\! \!\!\!\!\!\!
%\begin{array}{ll}
%&\frac{1}{\varepsilon} \int_{t}^{t+ \varepsilon} (t + \varepsilon-s)  [M(s)S^T(s) - I] ds \\
%&=  - \sum_{i =1}^n    {   \left( \frac{1}{\varepsilon} \int_{t}^{t+ \varepsilon}(t + \varepsilon-s)   \cos( 2 \omega_i s)  ds \right)}  e_ie_i^T \\
%& \quad +  \sum_{1 \leq i \neq j \leq n}  \frac{2a_j}{a_i}  \\
%& \qquad  \times \left( \frac{1}{\varepsilon} \int_{t}^{t+ \varepsilon}(t + \varepsilon-s)  \sin(\omega_is) \sin(\omega_js)  ds \right)  e_ie_j^T.
%\end{array}
%\end{equation}

\begin{equation}\label{Eq-rho1-2}
\!\!\!\!\!\!\!\!
\begin{array}{ll}
& \frac{\varepsilon}{T} \sum_{s=j}^{j+ T-1} (j+T-s)  [M(s)S^T(s) -  I]   \\
&=    \sum_{i=1 }^n    \Big( \frac{\varepsilon}{T} \sum_{s=j}^{j+ T-1} (j+T-s)  \cos(2 \omega_i s) \Big) e_{i}^T e_{i} \\
%& +  \sum_{i \neq l}^n  \frac{a_l}{a_i}   \Big( \frac{\varepsilon}{T} \sum_{s=j}^{j+ T-1} (j+T-s)  \cos(  [\omega_i-\omega_l] s) \Big)   e_{i}^T e_{l}   \\
%& - \sum_{i \neq l}^n  \frac{a_l}{a_i}    \Big( \frac{\varepsilon}{T} \sum_{s=j}^{j+ T-1} (j+T-s)  \cos(  [\omega_i+\omega_l] s) \Big)   e_{i}^T e_{l} \\
& +  \sum_{i \neq l}^n  \frac{a_l}{a_i}   \Big( \frac{\varepsilon}{T} \sum_{s=j}^{j+ T-1} (j+T-s) \\
& \qquad \quad  \times \big( \cos(  [\omega_i-\omega_l] s) - \cos(  [\omega_i+\omega_l] s) \Big) \Big)   e_{i}^T e_{l}.
\end{array}
\end{equation}
Using trigonometric identities and summation by parts we obtain
\begin{equation}\label{Ineq-rho1-3}
 \!\!\!\! \!\!\!\!\!\!
\begin{array}{ll}
& || \frac{\varepsilon}{T} \sum_{s=j}^{j+ T-1} (j+T-s)  [M(s)S^T(s) -  I]   || \\
& \leq   0.19245   n D_M     \Big[ \sum_{i=1 }^n   \frac{1 }{ i } +   \sum_{i \neq l}^n  \frac{a_l}{a_i}  [   \frac{ 1 }{ |i-l| }   +   \frac{ 1 }{ i+l }  ]  \Big]  \sqrt{\varepsilon}.
\end{array}
\end{equation}
Thus, from \eqref{Eq-rho1-1}-\eqref{Ineq-rho1-3}, we get
\begin{equation}\label{Ineq-rho1-4}
 ||\rho_{1}(j)||   \leq   \sqrt{\varepsilon} \bar{\rho}_1, \quad j \geq 0,
\end{equation}
with $\bar{\rho}_1$ defined in \eqref{Eq-rho-gamma-bounds}. Also, the function $\rho_{2}(j)$ can be rewritten as
\begin{equation}\label{Eq-rho2-1}
	\rho_{2}(j)  =    \sum_{i=1}^n \frac{2k}{a_i}  \left[ - \frac{\varepsilon}{T} \sum_{s=j}^{j+T-1}   (j+T-s)  \sin( \omega_i s)  \right]  e_i.
\end{equation}
Using summation by parts, we have
\begin{equation}\label{Eq-rho2-2}
\!\!\!\!\!\!\! \!\!
\begin{array}{ll}
&  - \frac{\varepsilon}{T} \sum_{s=j}^{j+T-1}   (j+T-s)  \sin( \omega_i s)     = -\frac{  \varepsilon \cos \left(  \frac{(2j-1)\omega_i}{2}\right) }{ 2   \sin (\frac{\omega_i}{2})}.
\end{array}
\end{equation}
Thus, from \eqref{Eq-rho2-1} and \eqref{Eq-rho2-2}, we get
\begin{equation}\label{Ineq-rho2-3}
 |\rho_{2}(j)|   \leq   \sqrt{\varepsilon}  \bar{\rho}_2, \quad j \geq 0,
\end{equation}
with $\bar{\rho}_2$  defined in \eqref{Eq-rho-gamma-bounds}. %In a similar manner to $\rho_2(j)$, an upper bound for $\rho_4(j)$ can be determined.
An upper bound on $\rho_4(j)$ can be derived similarly to the bound on  $\rho_2(j)$.

The function $\rho_{3}(j)$ can be rewritten as
\begin{equation}\label{Eq-rho3-1}
\!\!\!\!\!\!\!\!\!\!
\begin{array}{ll}
&\rho_{3}(j) =     \sum_{i=1}^n \frac{k}{a_i}  \Big[ \frac{\varepsilon}{T}  \sum_{s=j}^{j+T-1}(j+T-s) \\
& \qquad \qquad \qquad  \qquad  \quad   \times \sin( \omega_i s) |S(s)|_{H}^2 \Big]    e_i.
\end{array}
\end{equation}
Using discrete Cauchy-Schwarz inequality, we obtain
\begin{equation}\label{Ineq-rho3-2}
\!\!\!\!\!\!\! \!\!\!
\begin{array}{ll}
 &  | \frac{\varepsilon}{T}  \sum_{s=j}^{j+T-1}(j+T-s)\sin( \omega_i s) |S(s)|_{H}^2 |  \\
 & \leq  \left[ \frac{\varepsilon}{T} \sum_{s=j}^{j+T-1}(j+T-s)^2 \sin^2( \omega_i s) \right]^{\frac{1}{2}}   \\
 & \qquad \qquad \quad \times \left[   \frac{\varepsilon}{T} \sum_{s=j}^{j+T-1} |S^T(s)HS(s)|^2   \right]^{\frac{1}{2}}.
\end{array}
\end{equation}
Repeated summation by parts yields the following bound for the first term on the right side of \eqref{Ineq-rho3-2}:
\begin{equation}\label{Ineq-rho3-3}
\!\!\!\!\!\!\!\!\!\!\!
\begin{array}{ll}
& |{  \frac{1}{\varepsilon} \int_{t}^{t+ \varepsilon} (t + \varepsilon-s)^2 \sin^2(\omega_i s)   ds } |   \leq     \\
&  \qquad  \frac{ ( n  D_M + \sqrt{\varepsilon})(2n  D_M    + \sqrt{\varepsilon})}{12}  \\
&  \qquad \qquad \quad  + \frac{ \sqrt{\varepsilon} ( \sqrt{\varepsilon} + n D_M   )  + D_M^2 n^2( 1+   \frac{ 1 }{3.3072 \pi  i }   )  }{3.3072 \pi i }  .
\end{array}
\end{equation}
The second term on the right side of \eqref{Ineq-rho3-2} can be bounded as
\begin{equation}\label{Ineq-rho3-4}
\!\!\!\!\!\!\!
\begin{array}{ll}
\frac{\varepsilon}{T} \sum_{s=j}^{j+T-1} (S^T(s)HS(s))^2   \leq     \varepsilon H_M^2  \left( \sum_{i=1}^{n}a_i^2\right) ^2.
\end{array}
\end{equation}
Finally, from \eqref{Eq-rho3-1}-\eqref{Ineq-rho3-4} we have
%\begin{equation}\label{Ineq-rho3-5}
 $|\rho_{3}(j)|   \leq   \sqrt{\varepsilon}  \bar{\rho}_3$
%\end{equation}
with $\bar{\rho}_3$ defined in \eqref{Eq-rho-gamma-bounds}.
$\hfill \Box$

%%%%%%%%%%%%%%%%%%%%%%%%%%%%%%%%%%%%%%%%%%%%%%%%%%%%%%%%%%%%%%%%%%%%%%%%%%%%%%%%%%%%%%%%%%%%%%%%%%%%%%%%%%%%%%%%%%%%%%%%%%%%%%%%%%%%%%%%%%%%%%%%%%%%%%%%%%%%%%%%%%%%%%%%%%%%%%%%%%%%%%%%%%%%%%%%

{ \textit{ Proof of Theorem \ref{THM-1}: }}
%\begin{proof}[Proof of Theorem \ref{THM-delay-1}]
Given $\sigma_0>0$, let $\sigma$ be subject to $\sigma_0< \sigma$. From \eqref{Eq-dervative-tildetheta-Un-1}, we have
\begin{equation} \label{Ineq-tildetheta-sigma_0-delay}
\!\!\!\!\!
\begin{array}{ll}
|\tilde{\theta}(j)|  = | {\theta}(0)-\theta^*  |  \leq  \sigma_{0}  \underbrace{ < }_{\eqref{Ineq-THM-Un-condition-3}} \sigma |\bar{\lambda}|^{D_M} \leq \sigma |\bar{\lambda}|^{  j },
\end{array}
\end{equation}
for $0 \leq j \leq D_M$. We first assume (and further prove) that
\begin{equation} \label{Ineq-tildetheta-sigma-delay}
	|\tilde{\theta}(j) | < \sigma |\bar{\lambda}|^{  j}, \quad   \forall j > D_M.
\end{equation}
Then, from \eqref{Ineq-tildetheta-sigma_0-delay} and \eqref{Ineq-tildetheta-sigma-delay}, we have
\begin{equation}\label{Ineq-tildetheta-sigma-all-delay}
|\tilde{\theta}(j)| < \sigma |\bar{\lambda}|^j,  \quad   j  \geq 0.
\end{equation}
By using \eqref{Ineq-tildetheta-sigma-all-delay}, it follows from \eqref{Eq-Quadratic-Form}, \eqref{Eq-y-Q} and \eqref{Eq-bounds-THM-1} that
\begin{equation} \label{Ineq-y-Q*-delay}
\!\!\!\!\!\! \!\!\!
\begin{array}{ll}
&  |y(j)-Q^*|  =   \frac{1}{2} [\tilde{\theta}( j - D(j)) \\
 &  \qquad \qquad \qquad \qquad  + \alpha(j- D(j))S(j- D(j))|_H^2   \\
 & \qquad \quad \leq  \frac{H_M}{2}  \left[  \sigma |\bar{\lambda}|^{ j - D_M} + \alpha_0 \sqrt{\sum_{i=1}^na_i^2} |\bar{\lambda}|^{ j - D_M} \right] ^2   \\
% & \qquad \qquad \leq  \frac{H_M}{2}  \left[  \sigma   + \alpha_0 \sqrt{\sum_{i=1}^na_i^2} \right]^2 |\bar{\lambda}| ^{ 2j - 2D_M}   \\
% & \qquad \quad \leq  \frac{H_M}{2} |\bar{\lambda}| ^{  - 2D_M}  \left[  \sigma   + \alpha_0 \sqrt{\sum_{i=1}^na_i^2} \right]^2 |\bar{\lambda}| ^{ 2j  }   \\
 &  \qquad \quad \leq \sigma_{y} |\bar{\lambda}|^{2 j}, \qquad \qquad \qquad  j \geq D_M.
  \end{array}
\end{equation} 
We use the variation of constants formula for equation \eqref{Eq-dervative-tildeeta-Un} to obtain
\begin{equation} \label{Eq-tildeeta-sol-delay}
\!\!\!\!\!\!\!\!
\begin{array}{ll}
	& \tilde{\eta}(j) = [1-\varepsilon\omega_h]^{j-D_M}\tilde{\eta}(D_M)  \\
	& \qquad \quad + \varepsilon \omega_h   \sum_{i=D_M}^{j-1} [1-\varepsilon\omega_h]^{j-1-i} [y(i)-Q^*],
  \end{array}
\end{equation}
for $j > D_M$. Employing \eqref{Ineq-tildetheta-sigma-all-delay},  \eqref{Ineq-y-Q*-delay} and \eqref{Eq-tildeeta-sol-delay}, we get
\begin{equation} \label{Eq-tildeeta-bound-delay}
\!\!\!\!\!
\begin{array}{ll}
&	|\tilde{\eta}(j) |   \leq   |1-\varepsilon\omega_h|^{j-D_M}|\tilde{\eta}(D_M)|   \\
 &  \qquad \qquad + \varepsilon \omega_h   \sum_{i=D_M}^{j-1} |1-\varepsilon\omega_h|^{j-1-i} |y(i)-Q^*|   \\
% & \leq     \Delta_Q  |1-\varepsilon\omega_h|^{j-D_M}  \\
% & + \varepsilon \omega_h \sigma_y  |1-\varepsilon\omega_h|^{j-1} \sum_{i=D_M}^{j-1}  |1-\varepsilon\omega_h|^{-i}  \bar{\lambda}^{2i}    \\
% & \leq     \Delta_Q  |1-\varepsilon\omega_h|^{j-D_M}  \\
% & + \varepsilon \omega_h \sigma_y  |1-\varepsilon\omega_h|^{j-1} \sum_{k=0}^{j-D_M-1}    ( \frac{\bar{\lambda}^{2} }{ |1-\varepsilon\omega_h| })^{k+D_M} \\
% & \leq     \Delta_Q  |1-\varepsilon\omega_h|^{j-D_M}  \\
% & + \varepsilon \omega_h \sigma_y  |1-\varepsilon\omega_h|^{j-1} ( \frac{\bar{\lambda}^{2} }{ |1-\varepsilon\omega_h| })^{D_M} \sum_{k=0}^{j-D_M-1} ( \frac{\bar{\lambda}^{2} }{ |1-\varepsilon\omega_h| })^{k} \\
% & \leq      \Delta_Q  |1-\varepsilon\omega_h|^{j-D_M}   \\
% & + \varepsilon \omega_h \sigma_y  |1-\varepsilon\omega_h|^{j-1}  ( \frac{\bar{\lambda}^{2} }{ |1-\varepsilon\omega_h| })^{D_M} \sum_{k=0}^{j-D_M-1} ( \frac{\bar{\lambda}^{2} }{ |1-\varepsilon\omega_h| })^k   \\
% & \leq     \Delta_Q  |1-\varepsilon\omega_h|^{j-D_M}  \\
 %& + \varepsilon \omega_h \sigma_y  |1-\varepsilon\omega_h|^{j-1}  ( \frac{\bar{\lambda}^{2} }{ |1-\varepsilon\omega_h| })^{D_M} \frac{ ( \frac{\bar{\lambda}^{2} }{ |1-\varepsilon\omega_h|})^{j-D_M} } {\frac{\bar{\lambda}^{2} }{ |1-\varepsilon\omega_h|}-1}   \\
%& \leq    \Delta_Q  |1-\varepsilon\omega_h|^{j-D_M}  +  \frac{   \varepsilon \omega_h \sigma_y  } { \bar{\lambda}^{2}- |1-\varepsilon\omega_h| }  \bar{\lambda}^{2j}  \\
& \qquad  \underbrace{\leq}_{ \eqref{Ineq-THM-Un-condition-2} }  \Delta_Q  |1-\varepsilon\omega_h|^{-D_M} |\bar{\lambda}|^{2j}  +  \frac{   \varepsilon \omega_h \sigma_y  } { \bar{\lambda}^{2}-|1-\varepsilon\omega_h| }  |\bar{\lambda}|^{2j} \\
& \qquad < \sigma_{\eta}  |\bar{\lambda}|^{2j}, \quad      j \geq D_M.
\end{array}
\end{equation}
In addition, via \eqref{Eq-G-delay}, \eqref{Eq-Delta}, \eqref{Eq-Y-delay} and \eqref{Ineq-tildetheta-sigma-all-delay}, we have
\begin{equation} \label{Ineq-G-Y-bound-delay}
\begin{array}{ll}
& |G(j)| < \sqrt{\varepsilon} \Delta_{G} |\bar{\lambda}|^{j}, \quad |Y(j)| < (\sqrt{\varepsilon})^3 \Delta_{Y} |\bar{\lambda}|^{j}, \\
& |\Delta(j)| <  \Delta |\bar{\lambda}|^{j}, \quad  |\bar{\Delta}_{\sqrt{\varepsilon}}(j)| < \sqrt{\varepsilon} \bar{\Delta}_{out} |\bar{\lambda}|^{j},
\end{array}
\end{equation}
for $j \geq D_M$. Using the variation of constants formula for \eqref{Eq-dervative-z-Y-delay}, we obtain
\begin{equation} \label{Eq-z-sol-delay}
\begin{array}{ll}
& z(j) =   [I -\varepsilon k H]^{j-D_M}z(D_M) \\
& \quad +   \sum_{i=D_M}^{j-1} [I -\varepsilon k H]^{j-1-i} Y(i), \quad j > D_M.
	\end{array}
\end{equation}
Employing \eqref{Ineq-tildetheta-sigma-all-delay} and \eqref{Ineq-G-Y-bound-delay}, we get from \eqref{Eq-z-sol-delay} that
\begin{equation} \label{Ineq-z-bound-delay}
\begin{array}{ll}
	|z(j)|  & \leq |1-\varepsilon k H_m|^{j-D_M}|z(D_M)| \\
 &  \qquad  +   \sum_{i=D_M}^{j-1} |1-\varepsilon k H_m|^{j-1-i} |Y(i)|    \\
 & \underbrace{\leq}_{\eqref{Ineq-THM-Un-condition-2}}  \Big[  ( \sigma_0 +\sqrt{\varepsilon} \Delta_G |\bar{\lambda}|^{D_M}) |1-\varepsilon k H_m|^{-D_M} \\
 &  \qquad \qquad \qquad \qquad   + \frac{ (\sqrt{\varepsilon})^3 \Delta_Y  }{ |\bar{\lambda}| - |1-\varepsilon k H_m|}   \Big] |\bar{\lambda}|^{j},
 \end{array}
\end{equation}
for $j \geq D_M$. Then it follows from \eqref{Ineq-G-Y-bound-delay} and \eqref{Ineq-z-bound-delay} that
\begin{equation} \label{Ineq-tildetheta-bound-delay}
\!\!\!\!
\begin{array}{ll}
  |\tilde{\theta}(j)|    & \leq     \sqrt{\varepsilon} \Delta_G |\bar{\lambda}|^{j} \\
 & \quad +  \Big[  ( \sigma_0 +\sqrt{\varepsilon} \Delta_G |\bar{\lambda}|^{D_M}) |1-\varepsilon k H_m|^{-D_M} \\
 &  \qquad \qquad \qquad \qquad \qquad  + \frac{ (\sqrt{\varepsilon})^3 \Delta_Y  }{|\bar{\lambda}| - |1-\varepsilon k H_m| }   \Big] |\bar{\lambda}|^{j}  \\
 &   \underbrace{<}_{\eqref{Ineq-THM-Un-condition-3}} \sigma |\bar{\lambda}|^{j}, \qquad j \geq D_M.
\end{array}
\end{equation}
Also, one can easily show that inequalities \eqref{Ineq-THM-Un-condition-1}-\eqref{Ineq-THM-Un-condition-3} hold for small enough $\varepsilon^*$.

 We prove further that inequalities  \eqref{Ineq-THM-Un-condition-1}-\eqref{Ineq-THM-Un-condition-3} guarantee the bound \eqref{Ineq-tildetheta-sigma-delay}. From \eqref{Ineq-THM-Un-condition-1}-\eqref{Ineq-THM-Un-condition-3}, \eqref{Ineq-tildetheta-sigma_0-delay} the inequality $|\tilde{\theta}(j) | < \sigma |\bar{\lambda}|^j$ holds for $0 \leq j \leq D_M$. %Then $|\tilde{\theta}(j) | <  \sigma \bar{\lambda}^j$  holds also for some $t > D_M$ due to continuity of $\tilde{\theta}(j)$.
We assume by contradiction that there exists $j > 0$ such that \eqref{Ineq-tildetheta-sigma-delay} does not hold. Namely, there exists the smallest $j^* > 0$ such that $|\tilde{\theta}(j^*)| \geq \sigma |\bar{\lambda}|^{j^*}$ and $|\tilde{\theta}(j)| < \sigma |\bar{\lambda}|^j$ when $0 \leq j < j^*$.% Thus, $|\tilde{\theta}(j)| < \sigma \bar{\lambda}^j$ holds for all $T-1 \leq j \leq j^*-1$.

Thus, we find that inequality \eqref{Ineq-z-bound-delay} holds for $j=j^*$. Then, by using inequalities \eqref{Ineq-THM-Un-condition-1}-\eqref{Ineq-THM-Un-condition-3}, we have
\begin{equation} \label{Ineq-bound-proof-delay-1}
\!\!\!\!\!\!\!\!\!\!\!\!
 \begin{array}{ll}
 & |\tilde{\theta}(j^*)| \leq     \sqrt{\varepsilon} \Delta_G |\bar{\lambda}|^{j^*} \\
 & \qquad\quad +  \Big[  ( \sigma_0 +\sqrt{\varepsilon} \Delta_G |\bar{\lambda}|^{D_M}) |1-\varepsilon k H_m|^{-D_M} \\
 &  \qquad \qquad \qquad     + \frac{ (\sqrt{\varepsilon})^3 \Delta_Y  }{|\bar{\lambda}| - |1-\varepsilon k H_m| }   \Big] |\bar{\lambda}|^{j^*} < \sigma |\bar{\lambda}|^{j^*},
\end{array}
\end{equation}
which contradicts to $|\tilde{\theta}(j^*)| \geq \sigma$ and completes the proof.% of \eqref{Ineq-THM-delay}.
%\end{proof}
$\hfill \Box$

%%%%%%%%%%%%%%%%%%%%%%%%%%%%%%%%%%%%%%%%%%%%%%%%%%%%%%%%%%%%%%%%%%%%%%%%%%%%%%%%%%%%%%%%%%%%%%%%%%%%%%%%%%%%%%%%%%%%%%%%%%%%%%%%%%%%%%%%%%%%%%%%%%%%%%%%%%%%%%%%%%%%%%%%%%%%%%%%%%%%%%%%%%%%%%%%
\footnotesize
\bibliographystyle{abbrv}
\bibliography{Reference}

%%%%%%%%%%%%%%%%%%%%%%%%%%%%%%%%%%%%%%%%%%%%%%%%%%%%%%%%%%%%%%%%%%%%%%%%%%%%%%%%%%%%%%%%%%%%%%%%%%%%%%%%%%%%%%%%%%%%%%%%%%%%%%%%%%%%%%%%%%%%%%%%%%%%%%%%%%%%%%%%%%%%%%%%%%%%%%%%%%%%%%%%%%%%%%%%

\end{document}